\documentclass{fu-agplus-proceedings}
\usepackage[utf8]{inputenc}
\usepackage{tikz-cd}
\usepackage{float}
\usetikzlibrary{calc, positioning,decorations.markings,arrows}

\begin{document}
\proctitle        {Moduli of Differentials and Teichm\"uller Dynamics}
\proclecturer     {Andrei Bud, Dawei Chen}
\procinstitution  {Humboldt Universit\"at zu Berlin, Boston College}
\procmail         {andreibud95@gmail.com, dawei.chen@bc.edu}
%\procauthor       {Andrei Bud}
\procabstract     {%
	An increasingly important area of interest for mathematicians is the study of Abelian differentials. This growing interest can be attributed to the interdisciplinary role this subject plays in modern mathematics, as various problems of algebraic geometry, dynamical systems, geometry and topology lead to the study of such objects. It comes as a natural consequence that we can employ in our study algebraic, analytic, combinatorial and dynamical perspectives. These lecture notes aim to provide an expository introduction to this subject that will emphasize the aforementioned links between different areas of mathematics. We will associate to an Abelian differential a flat surface with conical singularities such that the underlying Riemann surface is obtained from a polygon by identifying edges with one another via translation. We will focus on studying these objects in families and describe some properties of the orbit as we vary the polygon by the action of $GL_2^{+}(\mathbb{R})$ on the plane.   
}

\makeheader
%\maketitle

%\frontmatter

%\tableofcontents

%\mainmatter
\section{Introduction}
These represent the lecture notes for the three lectures taught by Prof. Dr. Dawei Chen at FU Berlin on the occasion of the Math+ Fall School in Algebraic Geometry held from September 30 to October 4, 2019. \\

The main topic discussed in these lectures is about translation surfaces. They can be viewed from two seemingly unrelated perspectives: from a geometric or from a complex analytic point of view. We will concentrate our attention on the first one. The interested reader can find more detailed introductions on this subject in \cite{abche17} or \cite{abwri15}, from the perspective of Algebraic Geometry and respectively Dynamical Systems. For a broader survey of the literature, we refer to \cite{abzor06}. We hope that these lecture notes will captivate the imagination of young mathematicians and lay the case for the importance of this topic. \\

 We will start by providing the definition of a translation surface and some  equivalent perspectives on the subject. We will continue by studying these surfaces in families stratified by the type of zeroes. In particular, we describe the dimensions and connected components of the corresponding strata, and explain why some strata are disconnected. We then use the polygonal description of a translation surface to motivate the existence of an action of $GL_2^{+}(\mathbb{R})$ on the strata, and provide some fundamental theorems to describe the analytic closures of such orbits. We will end these notes with the exercise sheet, together with some sketch of solutions.\\

\textbf{Acknowledgements.} We would like to thank the organizers Daniele Agostini, Thomas Kr\"amer, Marta Panizzut and Rainer Sinn for this wonderfull and stimulating fall school. The lecturer is also grateful for the support offered by the National Science Foundation CAREER Awards through grant DMS-1350396. 

\section{Translation Surfaces} 
In this section, we will provide different characterizations of translation surfaces and consider some examples to help us familiarize with the new concepts. We start by stating the following definition:
\begin{defn} \label{ab:def1}
	Let X be a real surface (assumed to be connected, compact and orientable) and let $p_1,...,p_n \in X $ be special points satisfying the following conditions: \begin{bulist}
	\item Away from $p_1,...,p_n$ we have that $X \setminus  \left\{ p_1,...,p_n \right\} $ is "locally flat", i.e. covered by finitely many charts with transition functions given by translations. 
	\item At each $p_i$, under the induced flat metric, the angle at the point $p_i$ is equal to $2\pi(m_i+1)$ for some $m_i \in \mathbb{Z}^{+}$. 
\end{bulist}
	We call a surface $X$, together with points $p_1,...,p_n$ as above, a translation surface. 
	
\end{defn} 

Our next definition, which we will see is equivalent to the preceding one, is the following: 

\begin{defn} \label{ab:def2}
	A translation surface is a nonidentically zero Abelian differential, i.e. a non-zero holomorphic one-form $\omega$ on a Riemann surface $X$. 
\end{defn}
To build a bridge between the two definitions we make the remarks: \begin{bulist} 
\item The special points $p_1,...,p_n$ are the zeroes of $\omega$,  
\item $m_i$ is the zero order of $\omega$ at $p_i$. \par
\end{bulist}
Locally at $p_i$ we have a neighbourhood that looks like: 
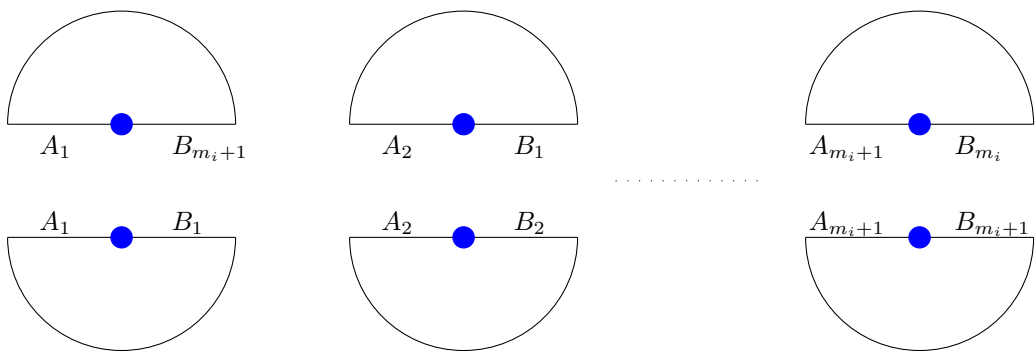
\begin{figure}[h] \centering
\begin{tikzpicture}[baseline=(current bounding box.north)] 

% A clipped circle is drawn
\begin{scope}

\draw (1.5,0) arc(0:180:1.5cm );
\draw (-1.5,0) -- (1.5,0);
\node[mark size=4pt,color=blue] at (0,0) {\pgfuseplotmark{*}};
\end{scope}
%
%%Labels for the vertices are typeset.
\node[below left= 0.5mm of {(-0.5,0)}] {$A_1$};
\node[below right= 0.5mm of {(0.5,0)}] {$B_{m_i+1}$};

% A clipped circle is drawn
\begin{scope}

\draw (-1.5,-1.5) arc(180:360:1.5cm );
\draw (-1.5,-1.5) -- (1.5,-1.5);
\node[mark size=4pt,color=blue] at (0,-1.5) {\pgfuseplotmark{*}};
\end{scope}
%
%%Labels for the vertices are typeset.
\node[below left= 0.5mm of {(-0.5,-1)}] {$A_1$};
\node[below right= 0.5mm of {(0.5,-1)}] {$B_1$};

% A clipped circle is drawn
\begin{scope}

\draw (6,0) arc(0:180:1.5cm );
\draw (3,0) -- (6,0);
\node[mark size=4pt,color=blue] at (4.5,0) {\pgfuseplotmark{*}};
\end{scope}
%
%%Labels for the vertices are typeset.
\node[below left= 0.5mm of {(4,0)}] {$A_2$};
\node[below right= 0.5mm of {(5,0)}] {$B_1$};

% A clipped circle is drawn
\begin{scope}

\draw (3,-1.5) arc(180:360:1.5cm );
\draw (3,-1.5) -- (6,-1.5);
\node[mark size=4pt,color=blue] at (4.5,-1.5) {\pgfuseplotmark{*}};
\end{scope}
\node[below left= 0.5mm of {(4,-1)}] {$A_2$};
\node[below right= 0.5mm of {(5,-1)}] {$B_2$};

\draw[loosely dotted] (6.5,-0.75) -- (8.5,-0.75);

% A clipped circle is drawn
\begin{scope}

\draw (12,0) arc(0:180:1.5cm );
\draw (9,0) -- (12,0);
\node[mark size=4pt,color=blue] at (10.5,0) {\pgfuseplotmark{*}};
\end{scope}
%
%%Labels for the vertices are typeset.
\node[below left= 0.5mm of {(10.2,0)}] {$A_{m_i+1}$};
\node[below right= 0.5mm of {(10.8,0)}] {$B_{m_i}$};

% A clipped circle is drawn
\begin{scope}

\draw (9,-1.5) arc(180:360:1.5cm );
\draw (9,-1.5) -- (12,-1.5);
\node[mark size=4pt,color=blue] at (10.5,-1.5) {\pgfuseplotmark{*}};
\end{scope}
\node[below left= 0.5mm of {(10.2,-1)}] {$A_{m_i+1}$};
\node[below right= 0.5mm of {(10.8,-1)}] {$B_{m_i+1}$};

\end{tikzpicture} 
\caption{Neighbourhood of a special point $p_i$}
\end{figure} \\

Here we glue together the radii with the same labelings. Observe that through this process all centers of the semidisks get identified. \\ 

Before proceeding to prove the equivalence between the two definitions, we will provide an example that will help us familiarize with the concepts.  
\begin{ex} \label{ab:h(2)} (translation surface as in Definition \ref{ab:def1})  
	We consider four vectors $v_1,v_2,v_3$ and $v_4$ in the plane and we take the polygon $X$ defined by them as suggested in the picture below.  \par
	
	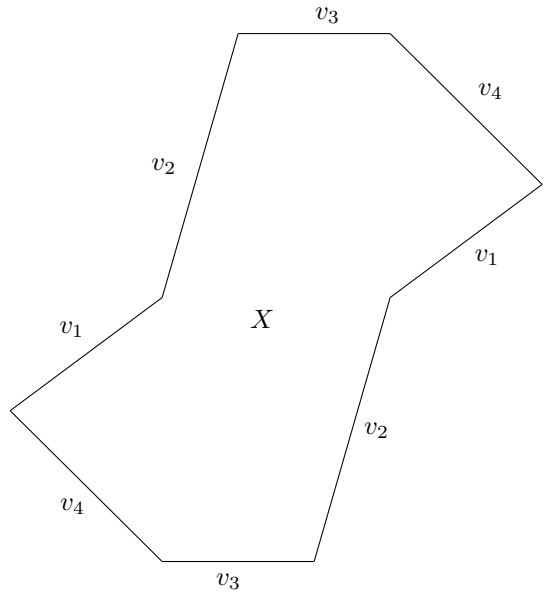
\begin{figure}[h] \centering \label{ab: fig2}
	\begin{tikzpicture}[baseline=(current bounding box.north)] \centering
	
	% A clipped circle is drawn
	\begin{scope}

	\draw (0,0) -- (2,1.5);
	\draw (2,1.5) -- (3,5);
	\draw (3,5) -- (5,5);
	\draw (5,5) -- (7,3);
	\draw (0,0) -- (2,-2);
	\draw (2,-2) -- (4,-2);
	\draw (4,-2) -- (5,1.5);
	\draw (5,1.5) -- (7,3);
	\node[below left= 0mm of {(1.1,1.3)}] {$v_1$};
	\node[below right= 0.5mm of {(1.7,3.5)}] {$v_2$};
	\node[below left= 0.5mm of {(4.5,5.5)}] {$v_3$};
	\node[below right= 0.5mm of {(6,4.5)}] {$v_4$};
	\node[below left= 0.5mm of {(6.6,2.3)}] {$v_1$};
	\node[below right= 0.5mm of {(4.5,0)}] {$v_2$};
	\node[below left= 0.5mm of {(3.2,-2)}] {$v_3$};
	\node[below right= 0.5mm of {(0.5,-1)}] {$v_4$};
	\node[below right= 0.5mm of {(3,1.5)}] {$X$};

	\end{scope}
	\end{tikzpicture} \caption{Genus 2 translation surface}
	\end{figure}
We identify each pair of $v_i,v_i$ by translation. We see that there are some special points in the picture: the vertices of the polygon. It can be easily seen, by chasing the identifications, that all vertices are identified with each other and correspond to the same point $p$.\\

	At $p$, the angle is $6\pi = (2\pi) \cdot 3$ and $X$ has genus $2$ because of the topological Euler characteristic: \begin{abclist}
	 \item one vertex: $(p)$ 
	 \item four edges: $(v_1,v_2,v_3,v_4)$ 
	 \item one face: $(X)$ 
	 \end{abclist}
	\[ 1+1-4 = 2 -2g \Rightarrow g =2.  \] \par
	We remark that points on the edges of the polygon may look special. However, the angle around any such point that is not a vertex is $2\pi$, and hence the only special points are the vertices. \par
\end{ex}
\begin{prop}
	The two definitions of a translation surface are equivalent. 
\end{prop}
\begin{proof}
	"Definition \ref{ab:def1} $\Rightarrow$ Definition \ref{ab:def2}: \\
	A consequence of our description would be that $X$ inherits a complex structure from $\mathbb{C}$. 
	We can define a one-form $\omega$ as: \begin{bulist} 
	\item Away from $p_i$, define $\omega = dz$ where $z$ is a local coordinate for a chart as in the definition of a translation surface. Because the transition functions are all translations, it follows that $\omega$ does not depend on the choice of such a local coordinate, and hence it is well-defined.
	\item At a special point $p_i$, there exists a local coordinate $u$, unique up to multiplication by a $(m_i+1)$-th root of unity, such that $u(p_i) = 0$ and $\omega = (m_i+1)u^{m_i}du$. We can think of $\omega$ as the pullback of $dz$ under the covering map $z = u^{m_i+1}$ which is a local isometry onto its image except at $p_i$. \end{bulist}
	\par
	This description allows us to extend the one-form to the points $p_1,...,p_n$ with zero orders $m_1,...,m_n$ correspondingly. \\ 
	 
	"Definition \ref{ab:def2} $\Rightarrow$ Definition \ref{ab:def1}": \\
	For this implication we provide a way to construct such charts as in the first definition. We get a local coordinate chart by fixing a point $p$ of $X$ and integrating along a curve with endpoints $p$ and $x$. The fact that the integral of $\omega$ is an invariant along homotopic curves implies that indeed the transition functions are translations. We leave to the reader the task of describing how to correctly choose such curves such that we indeed get local charts through this construction. 
	
\end{proof}
\begin{remark}
	In Example \ref{ab:h(2)}, the one-form $\omega$ looks like $d(u^3) = 3u^2du$ near the point $p$. In particular it follows that $\omega$ has a double zero at $p$. \par
\end{remark}
\par 
We are mainly interested in the geometric description using polygons and glueings. Hence it would make sense to state another definition of a translation surface that better fits in our study. \\
\begin{defn} \label{ab:def3}
	A translation surface is, up to equivalence, a finite union of polygons in $\mathbb{C}$ together with a choice of pairing for equal and parallel edges via translation. The edges that are paired and glued are required to be on "different sides" of the respective polygon. Two such unions define the same translation surface if we can cut one of them along straight lines and the pieces can be translated and glued together to obtain the second union of polygons. \par
	We remark that when a polygon is cut in two pieces, the newborn edges will be glued together. 
\end{defn} \ \\

We state the following proposition whose slogan would be: "We do not lose the generality when we restrict our attention to collections of polygons with glued edges". For a sketch of proof of this proposition we refer the reader again to \cite{abwri15}.
\begin{prop}
	Definition \ref{ab:def3} is equivalent to the previous definitions of a translation surface. 
\end{prop}

\begin{ex} \label{ab:h(1,1)}
	$\bullet$ $g = 2$ (with two simple zeroes)
	Consider a polygon defined by five vectors $v_1,v_2,v_3,v_4$ and $v_5$ as in the following picture: \\ 
	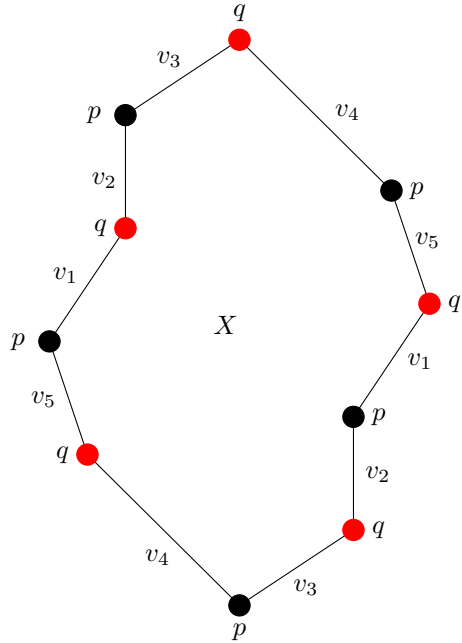
\begin{figure}[H] \centering
	\begin{tikzpicture}[baseline=(current bounding box.north)]
	% A clipped circle is drawn
	\begin{scope}

	\draw (0,0) -- (1,1.5);
	\draw (1,1.5) -- (1,3);
	\draw (1,3) -- (2.5,4);
	\draw (2.5,4) -- (4.5,2);
	\draw (4.5,2) -- (5,0.5);
	\draw (5,0.5) -- (4,-1);
	\draw (4,-1) -- (4,-2.5);
	\draw (4,-2.5) -- (2.5,-3.5);
	\draw (2.5,-3.5) -- (0.5,-1.5);
	\draw (0.5,-1.5) -- (0,0);
	\node[below left= 0mm of {(0.5,1.1)}] {$v_1$};
	\node[below right= 0.5mm of {(0.4,2.4)}] {$v_2$};
	\node[below left= 0.5mm of {(1.9,4)}] {$v_3$};
	\node[below right= 0.5mm of {(3.6,3.3)}] {$v_4$};
	\node[below left= 0.5mm of {(5.3,1.6)}] {$v_5$};
	\node[below left= 0.5mm of {(5.2,0)}] {$v_1$};
	\node[below right= 0.5mm of {(4,-1.5)}] {$v_2$};
	\node[below left= 0.5mm of {(3.7,-3)}] {$v_3$};
	\node[below right= 0.5mm of {(1.1,-2.6)}] {$v_4$};
	\node[below right= 0.5mm of {(-0.4,-0.5)}] {$v_5$};
	\node[below right= 0.5mm of {(2,0.5)}] {$X$};
	\node[left= 1mm of {(-0.1,0)}] {$p$};
	\node[left= 1mm of {(0.9,3)}] {$p$};
	\node[right= 1.2mm of {(4.5,2)}] {$p$};
	\node[right= 1.2mm of {(4,-1)}] {$p$};
	\node[below= 1.2mm of {(2.5,-3.5)}] {$p$};
	\node[mark size=4pt,color=black] at (0,0)
	{\pgfuseplotmark{*}};
	\node[mark size=4pt,color=black] at (1,3)
	{\pgfuseplotmark{*}};
	\node[mark size=4pt,color=black] at (4.5,2)
	{\pgfuseplotmark{*}};
	\node[mark size=4pt,color=black] at (4,-1)
	{\pgfuseplotmark{*}};
	\node[mark size=4pt,color=black] at (2.5,-3.5)
	{\pgfuseplotmark{*}};
	\node[mark size=4pt,color=red] at (1,1.5)
	{\pgfuseplotmark{*}};
	\node[mark size=4pt,color=red] at (2.5,4)
	{\pgfuseplotmark{*}};
	\node[mark size=4pt,color=red] at (5,0.5)
	{\pgfuseplotmark{*}};
	\node[mark size=4pt,color=red] at (4,-2.5)
	{\pgfuseplotmark{*}};
	\node[mark size=4pt,color=red] at (0.5,-1.5)
	{\pgfuseplotmark{*}};
	\node[left= 1.2mm of {(1,1.5)}] {$q$};
	\node[above= 1.2mm of {(2.5,4)}] {$q$};
	\node[right= 1.2mm of {(5,0.5)}] {$q$};
	\node[right= 1.2mm of {(4,-2.5)}] {$q$};
	\node[left= 1.2mm of {(0.5,-1.5)}] {$q$};
	
	\end{scope}
	\end{tikzpicture} \caption{Genus 2 translation surface with two simple zeroes}
	\end{figure}

	\par
	We get after the identification of $v_i,v_i$ for $i = \overline{1,5}$ that the vertices glue to two distinct points $p$ and $q$.
	At $p$ and $q$ the angles are both $4\pi = (2\pi)(1+1)$. \par 
	In summary, this is a translation surface of genus $2$ with two simple zeroes of $\omega$ at $p$ and $q$.
\end{ex} \par
\begin{ex}
	$\bullet$ $g = 1$ \par
	\begin{figure}[H] \centering
		
		%parallelogram equivalent to torus
		%%%%%%%%%%%%%%%%%%%%%%%%%%%%everywhere flat torus
		\begin{tikzpicture}[baseline=(current bounding box.north)] 
		
		% A clipped circle is drawn
		\begin{scope}
		\draw (0.3,-0.3) -- (1.3,1.5);
		\draw (1.3,1.5) -- (3.8,1.5);
		\draw (3.8,1.5) -- (2.8,-0.3);
		\draw (2.8,-0.3) -- (0.3,-0.3);
		\node[below right= 0.5mm of {(2,2)}] {$v_2$};
		\node[below right= 0.5mm of {(1.4,-0.3)}] {$v_2$};
		\node[below right= 0.5mm of {(0.1,0.9)}] {$v_1$};
		\node[below right= 0.5mm of {(3.3,0.8)}] {$v_1$};
		\node[below right= 0.5mm of {(5,0.7)}] {$\iff$};
		\node[below right= 0.5mm of {(7.3,-0.8)}] {Everywhere flat torus};
		\end{scope}
		\useasboundingbox (8,-0.5) rectangle (11,1.5);
		\draw (9,0.5) ellipse (2 and 1);
		\begin{scope}
		\clip (9,-0.7) ellipse (2 and 1.66);
		\draw (9,1.96) ellipse (2 and 1.66);
		\end{scope}
		\begin{scope}
		\clip (9,1.96) ellipse (2 and 1.66);
		\draw (9,-0.94) ellipse (2 and 1.66);
		\end{scope}
		\end{tikzpicture}
		\caption{Genus one translation surface}
	\end{figure}
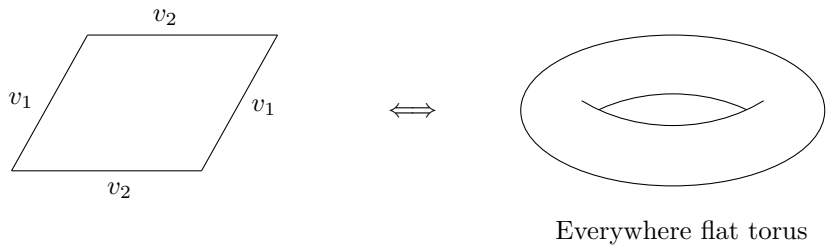
\end{ex}
The following is a well-known fact about Riemann surfaces. The reader interested to learn more about this and other basic properties about Riemann surfaces is invited to check \cite{abmir95}.
\begin{prop} \label{ab:sum2g-2} If $\omega$ is a non-trivial holomorphic one-form on a genus g Riemann surface X, then the total number of zeroes of $\omega$ (counting multiplicities) is $2g-2$.
\end{prop}
\section{Strata of Abelian Differentials}\ \par
We would be interested in considering families of Abelian differentials. This new perspective we want to follow is common in Algebraic Geometry and provides many advantages, especially as we are interested in the following two features: classification and parametrization of translation surfaces. \\

 As a consequence, the focus becomes the study of the geometry of the space $\mathcal{H}$ parametrizing up to isomorphism pairs $(X,\omega)$ where $X$ is a Riemann surface of genus $g$ and $\omega$ is a holomorphic one-form on $X$. The meaning of the word "isomorphism" in this context will be provided in Remark \ref{ab:iso}. This perspective, together with the precise definitions of the considered families can be found in \cite{abACG11}. \\ 
 
Proposition \ref{ab:sum2g-2} provides a restriction on the multiplicities of the zeroes of a one-form $\omega$. Accordingly, we will restrict our attention to this only possible case and define $\mu = (m_1,...,m_n)$ to be a positive partition of $2g-2$, i.e. $m_i \in \mathbb{Z}^{+}$ for every $i = \overline{1,n}$ and $m_1+m_2+...+m_n = 2g-2$. \\ 
\begin{defn}
	We define $\mathcal{H}(\mu)$ to be the subset of $\mathcal{H}$ parametrizing up to isomorphism pairs $(X,\omega)$ where $X$ is a Riemann surface of genus $g$ and $\omega$ is a one-form on $X$ of zero type $\mu$, i.e. $\omega$ has exactly $n$ zeroes $p_1,...,p_n$ such that at each $p_i$, $\omega$ has zero order $m_i$. \par
	The space $\mathcal{H}(\mu)$ is called the stratum of Abelian differentials of type $\mu$. \\
\end{defn} 
\begin{remark} \label{ab:iso}
	Two pairs $(X_1,\omega_1)$ and $(X_2,\omega_2)$ are called isomorphic if and only if there exists an isomorphism of Riemann surfaces $f\colon X_1\rightarrow X_2$ such that $f^{*}\omega_2 = \omega_1$. \\
\end{remark} 
\begin{remark}
	Example \ref{ab:h(2)} and Example \ref{ab:h(1,1)} are examples of elements of the strata $\mathcal{H}(2)$ and $\mathcal{H}(1,1)$ respectively.
\end{remark}

 \subsection{Period coordinates on $\mathcal{H}(\mu)$}
Take $(X,\omega)$ to be an element of $\mathcal{H}(\mu)$ and let 
$(\omega)_0 =  m_1p_1 + ...+m_np_n$ be the zero divisor associated to the one-form $\omega$. Consider $H_1(X,p_1,...,p_n;\mathbb{Z})$ be the first homology group relative to $p_1,...,p_n$. We know from Algebraic Topology that this group has rank $2g+n-1$. \\

It is important to understand where this number comes from as we will immediately see, this will help us understand the dimension of $\mathcal{H}(\mu)$. Being more precise the term $2g$ appears from a symplectic basis of the first homology group $H_1(X,\mathbb{Z})$ while the term $n-1$ comes from the relative part, i.e. the additional $n-1$ generators coming from curves connecting the point $p_1$ to $p_i$ for all $i = \overline{2,n}$. \\

Our goal is to prove that integrating $\omega$ along a basis of $H_1(X,p_1,...,p_n;\mathbb{Z})$ we get local coordinates on $\mathcal{H}(\mu)$. We will now specify the basis we are taking to define the local coordinates. 
\par 
We know that a genus $g$ Riemann surface can be obtained from a $4g$-gon by considering a glueing of the edges as in the next picture, where the order of the edges is $\gamma_1\gamma_2\gamma_1^{-1}\gamma_2^{-1}...\gamma_{2g-1}\gamma_{2g}\gamma_{2g-1}^{-1}\gamma_{2g}^{-1}$. Then, we consider the curves $\gamma_1,...,\gamma_{2g},\gamma_{2g+1},...,\gamma_{2g+n-1}$ as below. \\
\begin{figure}[H] \centering
\begin{tikzpicture}[baseline=(current bounding box.north)]

% A clipped circle is drawn
\begin{scope}

\draw (0,0) -- (2,0);
\draw (0,0) -- ($(0,0)+sqrt(2)*(-1,1)$);
\draw (2,0) -- ($(2,0)+sqrt(2)*(1,1)$);
\draw[dashed] ($sqrt(2)*(-1,1)$)-- ($(0,2)+sqrt(2)*(-1,1)$) ;
\draw[dashed] ($(2,0)+sqrt(2)*(1,1)$) -- ($(2,2)+sqrt(2)*(1,1)$);
\draw ($(2,2)+sqrt(2)*(1,1)$) -- ($(2,2)+sqrt(2)*(0,2)$);
\draw ($(0,2)+sqrt(2)*(-1,1)$) -- ($(0,2)+sqrt(2)*(0,2)$);
\draw ($(0,2)+sqrt(2)*(0,2)$) -- ($(2,2)+sqrt(2)*(0,2)$);

\draw  ($(0.5,0)+sqrt(2)*(0,1)$) to[out=0, in=-120]  ($(0,2)+sqrt(2)*(0,1)$);

\draw ($(0.5,0)+sqrt(2)*(0,1)$) to[out=0, in=-120]  ($(0,2)+sqrt(2)*(1,1)$);
\draw  ($(0.5,0)+sqrt(2)*(0,1)$)to[out=0, in=-120]   ($(2,2)+sqrt(2)*(0.5,1)$);
\node[below left= 0mm of {(-0.8,0.8)}] {$\gamma_1$};
\node[below left= 0mm of {(1,0)}] {$\gamma_2$};
\node[below left= 0mm of {(3.6,1)}] {$\gamma_1^{-1}$};
\node[below left= 0mm of {(3.7,4.5)}] {$\gamma_{2g-1}$};
\node[below left= 0mm of {(1.2,5.3)}] {$\gamma_{2g}$};
\node[below left= 0mm of {(-1,4.5)}] {$\gamma_{2g-1}^{-1}$};
\node[below left= 0mm of {(0.5,1.4)}] {$p_1$};
\node[above = 0mm of {(0,3.5)}] {$p_2$};
\node[above= 0mm of {(1.3,3.5)}] {$p_3$};
\node[above= 0mm of {(2.5,3.5)}] {$p_n$};
\node[mark size=2pt,color=black] at ($(0.5,0)+sqrt(2)*(0,1)$)
{\pgfuseplotmark{*}};
\node[mark size=2pt,color=black] at ($(0,2)+sqrt(2)*(0,1)$)
{\pgfuseplotmark{*}};
\node[mark size=2pt,color=black] at ($(0,2)+sqrt(2)*(1,1)$) 
{\pgfuseplotmark{*}};
\node[mark size=2pt,color=black] at ($(2,2)+sqrt(2)*(0.5,1)$)
{\pgfuseplotmark{*}};

\end{scope}
\end{tikzpicture} \caption{Basis of the relative homology}
\end{figure}
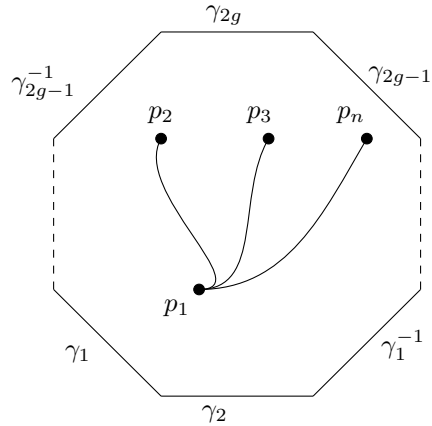
The first $2g$ curves form a symplectic basis of $H_1(X;\mathbb{Z})$ while the other $n-1$ represent the relative part. We have in fact that the $2g+n-1$ curves form a basis of $H_1(X,p_1,...,p_n;\mathbb{Z})$. \\

We claim that: 
\[ \Big(\int\limits_{\gamma_1}\omega,\int\limits_{\gamma_2}\omega,...,\int\limits_{\gamma_{2g+n-1}}\omega\Big) \]
provide local coordinates of $\mathcal{H}(\mu)$, called period coordinates. We hence have the following: 
\begin{cor}
	$\mathcal{H}(\mu)$ is a complex "manifold" of dimension $2g+n-1$. 
\end{cor}
\par  A heuristic argument for the claim is the following: We should really think of the period coordinates as the edges of the flat polygon representation corresponding to $\omega$, viewed as complex vectors. If we perturb these coordinates a little bit, we get a perturbation of the polygon by changing its shape a little bit. However, the glueings do not change under such a perturbation and hence the new translation surface we obtain lies in a neighbourhood of the point $(X,\omega)$ in $\mathcal{H}(\mu)$. \\

\begin{remark}
	The reason we put the word "manifold" under quotation marks is because this is actually not the exact statement of the results. In fact, $\mathcal{H}(\mu)$ is a complex orbifold of dimension $2g+n-1$, where the orbifold structure is due to automorphisms of special $(X,\omega)$. \\
\end{remark}
\begin{remark}
	There are many papers studying the dimension of the space $\mathcal{H}(\mu)$ or its variants. A formal proof of this dimension count can be found in \cite{abvee93}, Theorem $0.3$. Other proofs, coming from Deformation Theory can be found in \cite{abpol06}, \cite{abmo8} and \cite{abmon17}. Besides these articles answering the question of dimension for the strata of holomorphic Abelian differentials, there exist generalizations for meromorphic and higher order differentials. A dimension count using methods from analytic and flat geometry for the case of quadratic differentials can be found in \cite{abvee86}. An argument generalizing the result in \cite{abpol06} to the meromorphic case can be found in \cite{abfp18}, while a generalization of the arguments for higher differentials can be found in \cite{absch18}. Another proof treating the case of higher differentials can be found in \cite{abbcg19}. 
\end{remark}
\begin{ex} Take a translation surface as in Example \ref{ab:h(2)}: \\
	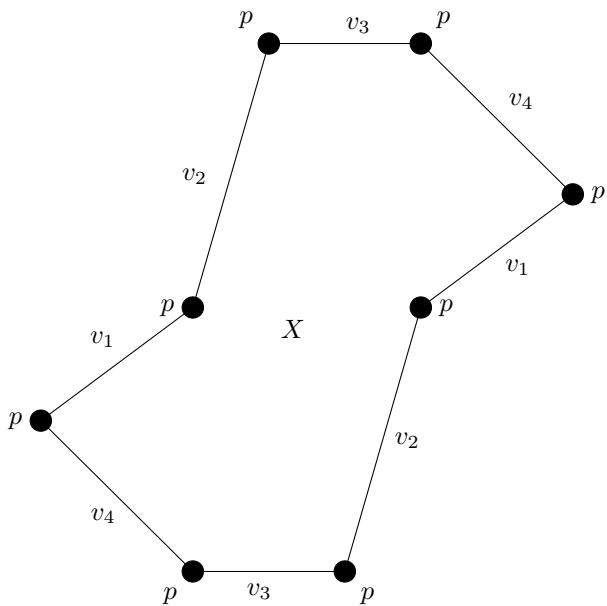
\begin{figure}[H] \centering
	\begin{tikzpicture}[baseline=(current bounding box.north)]
	
	% A clipped circle is drawn
	\begin{scope}

	\draw (0,0) -- (2,1.5);
	\draw (2,1.5) -- (3,5);
	\draw (3,5) -- (5,5);
	\draw (5,5) -- (7,3);
	\draw (0,0) -- (2,-2);
	\draw (2,-2) -- (4,-2);
	\draw (4,-2) -- (5,1.5);
	\draw (5,1.5) -- (7,3);
	\node[below left= 0mm of {(1.1,1.3)}] {$v_1$};
	\node[below right= 0.5mm of {(1.7,3.5)}] {$v_2$};
	\node[below left= 0.5mm of {(4.5,5.5)}] {$v_3$};
	\node[below right= 0.5mm of {(6,4.5)}] {$v_4$};
	\node[below left= 0.5mm of {(6.6,2.3)}] {$v_1$};
	\node[below right= 0.5mm of {(4.5,0)}] {$v_2$};
	\node[below left= 0.5mm of {(3.2,-2)}] {$v_3$};
	\node[below right= 0.5mm of {(0.5,-1)}] {$v_4$};
	\node[below right= 0.5mm of {(3,1.5)}] {$X$};
	\node[mark size=4pt,color=black] at (0,0)
	{\pgfuseplotmark{*}};
	\node[mark size=4pt,color=black] at (2,1.5)
	{\pgfuseplotmark{*}};
	\node[mark size=4pt,color=black] at (3,5)
	{\pgfuseplotmark{*}};
	\node[mark size=4pt,color=black] at (5,5)
	{\pgfuseplotmark{*}};
	\node[mark size=4pt,color=black] at (7,3)
	{\pgfuseplotmark{*}};
	\node[mark size=4pt,color=black] at (2,-2)
	{\pgfuseplotmark{*}};
	\node[mark size=4pt,color=black] at (4,-2)
	{\pgfuseplotmark{*}};
	\node[mark size=4pt,color=black] at (5,1.5)
	{\pgfuseplotmark{*}};
	\node[left = 1.2mm of {(0,0)}] {$p$};
	\node[left = 1.2mm of {(2,1.5)}] {$p$};
	\node[above left = 1.2mm of {(3,5)}] {$p$};
	\node[above right = 1.2mm of {(5,5)}] {$p$};
	\node[right = 1.2mm of {(7,3)}] {$p$};
	\node[right = 1.2mm of {(5,1.5)}] {$p$};
	\node[below right = 1.2mm of {(4,-2)}] {$p$};
	\node[below left = 1.2mm of {(2,-2)}] {$p$};

	\end{scope}
	\end{tikzpicture}\par \caption{Translation surface in the stratum $\mathcal{H}(2)$}
	\end{figure}
	In this case $(X,\omega) \in \mathcal{H}(2)$ and by Corollary $ 14 $ we know that the dimension of the stratum is $4$. 
	We see that the edges $v_1,v_2,v_3,v_4$ generate the homology group $H_1(X,p;\mathbb{Z})$. In particular, the period coordinates of $(X,\omega)$ are $v_1,v_2,v_3,v_4$ seen as complex numbers as: 
	\[v_i = \int\limits_{\gamma_i}dz \in \mathbb{C}.\]
\end{ex} 
Recall that we called $\mathcal{H}(\mu)$ the stratum of Abelian differentials of type $\mu$. We will now explain this terminology. The intuition behind is the following: When we take into account all the positive partitions of $2g-2$ and put them together, they provide a stratification of the total space of (nonzero) holomorphic one-forms. More precisely: 
\begin{center} 
	\begin{tikzcd}%
	\bigcup\limits_{\mu \vdash 2g-2}\mathcal{H}(\mu) = 	  &\mathcal{H}_g \setminus \left\{ 0 \right\} \arrow[d, ""]    \\
	& \mathcal{M}_g
	
	\end{tikzcd}%
\end{center} 
By $\mu \vdash 2g-2$ we mean that $\mu$ is a positive partition of $2g-2$. \par
Here $\mathcal{H}_g$ is the space of holomorphic one-forms on genus $g$ Riemann surfaces (regardless of multiplicity of zeroes as it is the union of all partitions), while $\mathcal{M}_g$ is the moduli space of genus $g$ Riemann surfaces. To understand this morphism, take $X \in \mathcal{M}_g$ and consider the fiber over it in $\mathcal{H}_g$. The fiber over $X$ is \[H^0(X,\Omega_X) = \left\{ \text{holomorphic one-forms on $X$} \right\} \cong \mathbb{C}^{g}.\] 
\par 
The space $\mathcal{H}_g$ can be seen as a holomorphic vector bundle of rank $g$ over $\mathcal{M}_g$. It is called the Hodge bundle of rank $g$ over $\mathcal{M}_g$. \\
\begin{ex}
	Take the partition $\mu = (1,1,...,1)\vdash 2g-2$. Then $\mathcal{H}(\mu)$ is open and dense in $\mathcal{H}_g$. \par
	
	It is well-known that the dimension over $\mathbb{C}$ of the moduli space $\mathcal{M}_g$ for $g\geq 2$ is equal to $3g-3$. This fact comes from Deformation Theory and it is related to the fact that there is a correspondence between the space of first order deformations of a curve $X$ and $H^1(X,T_X)$. \par 
	In particular, as $\mathcal{H}_g$ is a rank $g$ bundle over $\mathcal{M}_g$, it follows that \[\dim_\mathbb{C}\mathcal{H}_g = (3g-3)+g = 4g-3.\]
	As $\mathcal{H}(\mu)$ is an open and dense subset of $\mathcal{H}_g$ it follows that: 
	\[ \dim_\mathbb{C}\mathcal{H}(\mu) = 4g-3.\]
	Using Corollary 14, we see that 
	\[ \dim_\mathbb{C}\mathcal{H}(\mu) = 2g+n-1, \]
	but here $n = 2g-2$ and hence the results from the two ways of computing the dimension are consistent with each other. \par
\end{ex} \par 
Let us provide another intuition into why the dimension of $\mathcal{H}(\mu)$ is equal to $2g+n-1$ where $\mu = (m_1,...,m_n)$.\par 
Start with the space $\mathcal{H}(1,1,...,1)$ and view the differentials parametrized by $\mathcal{H}(m_1,...,m_n)$ as limits of differentials with simple zeroes when some of the zeroes merge. We expect that every merging of two zeroes of a differential $\omega$ will drop the dimension by 1. 
\begin{align*}
"\Rightarrow" \dim_\mathbb{C}\mathcal{H}(\mu)  &= \dim_\mathbb{C}\mathcal{H}(1,1,...,1) - (2g-2-n)\\
&= 4g-3 - (2g-2-n) \\
&=2g+n-1. 
\end{align*}
\par 
This calculation of the dimension is only meant to help us form an intuition into the relation between different strata. A formalization of this behaviour can be found in \cite{abkz03}, Proposition 4, which we now state: \\
\begin{prop}(Merging zeroes) \par
	Every connected component of the space $\mathcal{H}(m_1,...,m_n)$ contains in its closure a connected component of $\mathcal{H}(m_1+m_2,m_3,...,m_n)$. \\
\end{prop}
\begin{remark} We look at the map: \\
		\begin{tikzcd}\centering
		&\mathcal{H}(\mu) \arrow[d, ""] &\ \ \ \ \dim = 2g+n-1    \\
		& \mathcal{M}_g &  \dim = 3g-3
			\end{tikzcd} \\
By dimension reasons, if $2g+n-1 \leq 3g-3$, then a generic curve X does not admit any $\omega$ of type $\mu$. The equality case is included because we have a scaling factor for one-forms in $\mathcal{H}(\mu)$.
In fact, the dimension of the image in $\mathcal{M}_g$ is known for all positive partitions of $2g-2$. This characterization can be found in \cite{abgen18}, Theorem $5.7$. 
\end{remark} 
Before proceeding, let us take a minute and analyse such an example from the point of view of translation surfaces seen as polygons.
\begin{ex} Take the polygon from Example \ref{ab:h(1,1)}. This is a translation surface in the stratum $\mathcal{H}(1,1)$. We do the following thing: We shrink $v_5$ to $0$ and consider the limit case.  
	\begin{figure}[H] \centering
	\begin{tikzpicture}	% A clipped circle is drawn
	\begin{scope}
	\draw (0,0) -- (1,1.5);
	\draw (1,1.5) -- (1,3);
	\draw (1,3) -- (2.5,4);
	\draw (2.5,4) -- (4.5,2);
	\draw (4.5,2) -- (5,0.5);
	\draw (5,0.5) -- (4,-1);
	\draw (4,-1) -- (4,-2.5);
	\draw (4,-2.5) -- (2.5,-3.5);
	\draw (2.5,-3.5) -- (0.5,-1.5);
	\draw (0.5,-1.5) -- (0,0);
	\node[below left= 0mm of {(0.5,1.1)}] {$v_1$};
	\node[below right= 0.5mm of {(0.4,2.4)}] {$v_2$};
	\node[below left= 0.5mm of {(1.9,4)}] {$v_3$};
	\node[below right= 0.5mm of {(3.6,3.3)}] {$v_4$};
	\node[below left= 0.5mm of {(5.3,1.6)}] {$v_5$};
	\node[below left= 0.5mm of {(5.2,0)}] {$v_1$};
	\node[below right= 0.5mm of {(4,-1.5)}] {$v_2$};
	\node[below left= 0.5mm of {(3.7,-3)}] {$v_3$};
	\node[below right= 0.5mm of {(1.1,-2.6)}] {$v_4$};
	\node[below right= 0.5mm of {(-0.4,-0.5)}] {$v_5$};
	\node[below right= 0.5mm of {(2,0.5)}] {$X$};
	\node[left= 1mm of {(-0.1,0)}] {$p$};
	\node[left= 1mm of {(0.9,3)}] {$p$};
	\node[right= 1.2mm of {(4.5,2)}] {$p$};
	\node[right= 1.2mm of {(4,-1)}] {$p$};
	\node[below= 1.2mm of {(2.5,-3.5)}] {$p$};
	\node[mark size=4pt,color=black] at (0,0)
	{\pgfuseplotmark{*}};
	\node[mark size=4pt,color=black] at (1,3)
	{\pgfuseplotmark{*}};
	\node[mark size=4pt,color=black] at (4.5,2)
	{\pgfuseplotmark{*}};
	\node[mark size=4pt,color=black] at (4,-1)
	{\pgfuseplotmark{*}};
	\node[mark size=4pt,color=black] at (2.5,-3.5)
	{\pgfuseplotmark{*}};
	\node[mark size=4pt,color=red] at (1,1.5)
	{\pgfuseplotmark{*}};
	\node[mark size=4pt,color=red] at (2.5,4)
	{\pgfuseplotmark{*}};
	\node[mark size=4pt,color=red] at (5,0.5)
	{\pgfuseplotmark{*}};
	\node[mark size=4pt,color=red] at (4,-2.5)
	{\pgfuseplotmark{*}};
	\node[mark size=4pt,color=red] at (0.5,-1.5)
	{\pgfuseplotmark{*}};
	\node[left= 1.2mm of {(1,1.5)}] {$q$};
	\node[above= 1.2mm of {(2.5,4)}] {$q$};
	\node[right= 1.2mm of {(5,0.5)}] {$q$};
	\node[right= 1.2mm of {(4,-2.5)}] {$q$};
	\node[left= 1.2mm of {(0.5,-1.5)}] {$q$};
	%%\node[left= 1.2mm of {(6.4,0.5)}] {$\leadsto$};
	%%\node[above= 1mm of {(6.4,0.5)}] {$degeneration$};
	
	\end{scope}
	\end{tikzpicture} \caption{Translation surface in $\mathcal{H}(1,1)$ before shrinking $v_5$}
\end{figure}
	When $v_5$ becomes $0$, we get a translation surface of the form: \\
	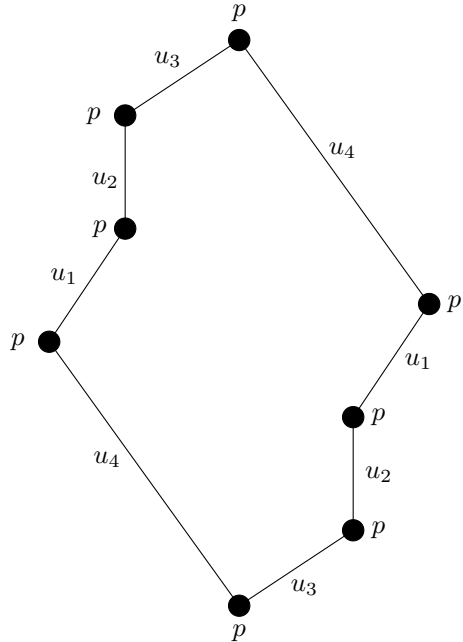
\begin{figure}[H] \centering
	\begin{tikzpicture}
	%%hereitende
	% A clipped circle is drawn
	\begin{scope}

	\draw (0,0) -- (1,1.5);
	\draw (1,1.5) -- (1,3);
	\draw (1,3) -- (2.5,4);
	\draw (2.5,4) -- (5,0.5);
	\draw (5,0.5) -- (4,-1);
	\draw (4,-1) -- (4,-2.5);
	\draw (4,-2.5) -- (2.5,-3.5);
	\draw (2.5,-3.5) -- (0,0);
	\node[left= 1mm of {(-0.1,0)}] {$p$};
	\node[left= 1mm of {(0.9,3)}] {$p$};
	\node[right= 1.2mm of {(4,-1)}] {$p$};
	\node[below= 1.2mm of {(2.5,-3.5)}] {$p$};
	\node[left= 1.2mm of {(1,1.5)}] {$p$};
	\node[above= 1.2mm of {(2.5,4)}] {$p$};
	\node[right= 1.2mm of {(5,0.5)}] {$p$};
	\node[right= 1.2mm of {(4,-2.5)}] {$p$};
	\node[mark size=4pt,color=black] at (0,0)
	{\pgfuseplotmark{*}};
	\node[mark size=4pt,color=black] at (1,3)
	{\pgfuseplotmark{*}};
	\node[mark size=4pt,color=black] at (4,-1)
	{\pgfuseplotmark{*}};
	\node[mark size=4pt,color=black] at (2.5,-3.5)
	{\pgfuseplotmark{*}};
	\node[mark size=4pt,color=black] at (1,1.5)
	{\pgfuseplotmark{*}};
	\node[mark size=4pt,color=black] at (2.5,4)
	{\pgfuseplotmark{*}};
	\node[mark size=4pt,color=black] at (5,0.5)
	{\pgfuseplotmark{*}};
	\node[mark size=4pt,color=black] at (4,-2.5)
	{\pgfuseplotmark{*}};
	\node[below left= 0mm of {(0.5,1.1)}] {$u_1$};
	\node[below right= 0.5mm of {(0.4,2.4)}] {$u_2$};
	\node[below left= 0.5mm of {(1.9,4)}] {$u_3$};
	\node[below left= 0.5mm of {(5.2,0)}] {$u_1$};
	\node[below right= 0.5mm of {(4,-1.5)}] {$u_2$};
	\node[below left= 0.5mm of {(3.7,-3)}] {$u_3$};
	\node[below left= 0.5mm of {(4.2,2.8)}] {$u_4$};
	\node[below left= 0.5mm of {(1.1,-1.3)}] {$u_4$};
		\end{scope}
	\end{tikzpicture} \caption{Degeneration to $\mathcal{H}(2)$}
	\end{figure}
	We see that, as the polygon degenerates, the points $p$ and $q$ get identified in the limit case. In particular, we can see from the picture that the limit is a translation surface parametrized by a point in $\mathcal{H}(2)$.
\end{ex} \par 
\subsection{Connected components of $\mathcal{H}(\mu)$}
We proceed to study the geometry of $\mathcal{H}(\mu)$. The next question we want to ask is if the space $\mathcal{H}(\mu)$ is connected, and if not: how many components can it have. This question is completely answered by two theorems in \cite{abkz03}. We will first state the result and then explain what it tells us. 
\begin{thm} The connected components of the stratum $\mathcal{H}(\mu)$ for genus $g \geq 4$ can be described as follows: \begin{bulist} 
	\item The stratum $\mathcal{H}(2g-2)$ has three connected components: the hyperelliptic component $\mathcal{H}^{hyp}(2g-2)$, the odd spin structure component $\mathcal{H}^{odd}(2g-2)$ and the even spin structure component $\mathcal{H}^{even}(2g-2)$; 
	\item If $g$ is odd, the stratum $\mathcal{H}(g-1,g-1)$ has three connected components as before: $\mathcal{H}^{hyp}(g-1,g-1)$, $\mathcal{H}^{odd}(g-1,g-1)$ and $\mathcal{H}^{even}(g-1,g-1)$;  
	\item All other strata of the form $\mathcal{H}(2l_1,...,2l_n)$ have only two connected components: $\mathcal{H}^{odd}(2l_1,...,2l_n)$ and $\mathcal{H}^{even}(2l_1,...,2l_n)$;  
	\item If $g$ is even then the stratum $\mathcal{H}(g-1,g-1)$ has two connected components: $\mathcal{H}^{hyp}(g-1,g-1)$ and $\mathcal{H}^{nonhyp}(g-1,g-1)$; 
	\item All other strata of holomorphic differentials $\mathcal{H}(\mu)$ are nonempty and connected. \end{bulist}
\end{thm}
In the same article there is also a description of these strata in low genus: 
\begin{thm} When the genus is $2$ or $3$ we have:\begin{bulist}
	\item When $g=2$, we have only two strata of holomorphic differentials: $\mathcal{H}(1,1)$ and $\mathcal{H}(2)$. Both of them are connected and coincide with the hyperelliptic component. 
	\item When $g=3$, the strata $\mathcal{H}(2,2)$ and $\mathcal{H}(4)$ have two connected components: the hyperelliptic one and the odd spin structure one. All other strata of holomorphic differentials are nonempty and connected. \end{bulist}
\end{thm}
We can make the following remarks: \begin{abclist}
 \item The stratum $\mathcal{H}(\mu)$ has at most three connected components. 
 \item Additional components appear because of the hyperelliptic and spin structures. \end{abclist}

It is a natural continuation to explain the meaning of the hyperelliptic and spin structures. We start by providing the definition of a hyperelliptic Riemann surface (and further proceed to explain what we mean when we say that the one-form $\omega$ on $X$ is hyperelliptic). \\

\begin{defn} A Riemann surface $X$ of genus $g$ is called hyperelliptic  if there exists a branched holomorphic cover $X \rightarrow \mathbb{P}^1$ of degree $2$. 
\end{defn}
\begin{remark}
	Using the Riemann-Hurwitz formula, in our particular case: 
	\[ 2g-2 = 2(2\cdot 0 - 2) + \sum_{p \in X}(e_p-1). \]
	Since the ramification index can take only the values $1$ and $2$, we deduce that the number of branch points on $\mathbb{P}^1$ is $2g-2+4 = 2g+2$. The double points of $X$ are called Weierstrass points. \par
	We see that a genus $g$ hyperelliptic curve $X$ has $2g+2$ Weierstrass points. 
\end{remark} \par 
Let $Hyp_g \subseteq \mathcal{M}_g$ be the sublocus parametrizing hyperelliptic Riemann surfaces. Using the fact that the monodromy around the branch points in $\mathbb{P}^1$ uniquely determines the Riemann surface $X$, and that the monodromy around every branch point must be a transposition in $S_2$, it follows that the branch points uniquely determine $X$. As a consequence we have that: 
\[ dim_\mathbb{C}Hyp_g = 2g+2-3 = 2g-1. \]
The drop in dimension by 3 is due to the fact that we need to factor out the automorphisms of $\mathbb{P}^1$. \par
As a consequence, it follows that $Hyp_g$ can be seen as an open subset of $Sym^{2g-1}\mathbb{P}^1$. In particular $Hyp_g$ is smooth, connected of dimension $2g-1$. \\ 

One of the advantages of working with hyperelliptic curves is that we have good explicit descriptions of them in terms of equations. We have the following description of a hyperelliptic curve and of its space of global holomorphic one-forms: 
\begin{prop}
	Let $X$ be a hyperelliptic Riemann surface and $a_1,...,a_{2g+2}$ its associated branch points in $\mathbb{P}^1$ under affine coordinates. Then $X$ can be realized as the zero locus: 
	\[x^2 = (z-a_1)...(z-a_{2g+2}) \text{ \ in \ } \mathbb{C}^2  \]
	completed to a Riemann surface. Then $\frac{dz}{x}, \frac{zdz}{x},...,\frac{z^{g-1}dz}{x}$ defined on this open subset in terms of $x$ and $z$, after extending to the entire $X$, are global holomorphic one-forms and they form a basis of $H^0(X,\Omega_X)$. \\
\end{prop}
We proceed to define what it means for a one-form to be hyperelliptic. But first, we want to digress and provide some references for what we did up to this point. 
\begin{remark}
	The description of a hyperelliptic curve and of its holomorphic one-forms, together with a discussion on the equivalence between monodromy and branch covers can be found in \cite{abmir95}. Another perspective on smoothness and dimension of the hyperelliptic locus can be found in \cite{abACG11}, Chapter 11, Lemma 6.15. \\
\end{remark} 
By applying an automorphism of $\mathbb{P}^1$, we can assume that $a_1 = 0$. Then near the point $(x,z) = (0,0)$ in $X$ we have: 
\[ x^2 = z\cdot(\text{polynomial nonvanishing at $z=0$}). \]
Denote by $q(z)$ this polynomial. By taking the holomorphic one-form $\frac{z^{g-1}dz}{x}$ and expressing it in terms of the local coordinate $x$ at 0 we have: 
\[\frac{z^{g-1}dz}{x} = \frac{x^{2g-1}dx}{x\cdot q(z)^{g-1}\cdot (q(z)+zq'(z))}.\] \\
In particular it follows that this holomorphic one-form has a unique zero of order $2g-2$ at $(0,0)$. In summary, there exists such a one-form for every Weierstrass point in $X$ (unique up to multiplication with a constant in $\mathbb{C}^{*}$). Such one-forms in the stratum $\mathcal{H}(2g-2)$ are called hyperelliptic. Similarly we can take a one-form of the type: $\frac{(z-a)^{g-1}dz}{x}$ for $a$ different from the $a_i$'s. This one-form has two zeroes, both of order $g-1$ and the two zeroes lie in the same fiber over $\mathbb{P}^1$, called conjugate points. The one-forms with two zeroes of order $g-1$ at two conjugate points are also called hyperelliptic in the stratum $\mathcal{H}(g-1,g-1)$. \\ 

Notice that beside the hyperelliptic one-forms we described, there is one that is hidden, namely $\frac{dz}{x}$. Its zeroes are the points at infinity which can be seen only when we complete the curve in $\mathbb{C}^2$ to a Riemann surface. \\
 
The hyperelliptic component $\mathcal{H}^{hyp}(2g-2) \subseteq \mathcal{H}(2g-2)$ is the subspace consisting of points $(X,\omega)$ where $X$ and $\omega$ are both hyperelliptic. \\

The dimension of both spaces is $2g$. The dimension on the left hand side is $2g$ because $\dim_\mathbb{C}Hyp_g = 2g-1$ and an additional $1$-dimension comes from the scaling factor of a one-form.

\par
The component $\mathcal{H}^{hyp}(g-1,g-1)$ of $\mathcal{H}(g-1,g-1)$ can be defined in a similar manner, with the obvious modifications. \\ 

Our next goal is to understand the spin structure. We start by providing the definition. 
\begin{defn}
	Take $\mu = (2k_1,...,2k_n)$ a positive partition of $2g-2$ with all its entries even numbers. Given any $(X,\omega)$ in $\mathcal{H}(\mu)$, denote by $p_i$ the zero of order $2k_i$ for $i=\overline{1,n}$.  
    Then $\dim_\mathbb{C}H^0(X,k_1p_1+...+k_np_n) \pmod{2}$ is called the parity of the spin/theta characteristic of $\omega$. \\
	
	Here we denoted $H^0(X,k_1p_1+...+k_np_n)$ to be the space of meromorphic functions $f$ on $X$ having the property that: 
	\[ div(f)+ \sum_{i=1}^nk_ip_i \geq 0. \] 
\end{defn}
Let us now explain what is so special about this parity. It is long time known that this value is constant on connected families of pointed curves. A proof of this result can be found in the papers \cite{abAti71} and \cite{abmum71}. We will state here the version of the theorem appearing in \cite{abAti71}, with the observation that it also has an algebraic counterpart.  \\
\begin{thm}(modulo $2$ stability) 
	Let $X_t$ be a holomorphic family of compact Riemann surfaces where $t$ is varying over the unit disk in $\mathbb{C}$, denoted $\Delta$. Consider $D_t$ a holomorphic family of divisors on the family $X_t$ such that:
	\[2D_t = K_t \text{ \ \ } \forall t \in \Delta, \] 
	where $K_t$ is the canonical divisor class of $X_t$. \par 
	Then $\dim_\mathbb{C}H^0(X_t,D_t)$ modulo $2$ does not depend on the parameter $t$. \\
\end{thm}
We will not focus much on this theorem. Our interest is simply the following slogan we derive from it: $\mathcal{H}(2k_1,...,2k_n)$ is split into two connected components by the modulo 2 stability. \\ 

The modulo 2 stability is suitable when working with pointed Riemann surfaces, but it is not what we are looking for when we are working with translation surfaces seen geometrically. For this, we will provide an equivalent definition of modulo 2 stability that is more accessible from this perspective. \\ 
 
Let us take a symplectic basis $\langle a_i,b_i\rangle_{i=1}^g$ of $H_1(X;\mathbb{Z})$ having the following properties with regard to intersection: \par 
\[a_i \cdot b_j =
\left\{
\begin{array}{ll}
0  & \mbox{if } i \neq j \\
1 & \mbox{if } i = j
\end{array}
\right.
\]
\[ a_i \cdot a_j = 0 \text{ \ \ } \forall i,j \in \left\{1,...,g \right\}  \]
\[ b_i \cdot b_j = 0 \text{ \ \ } \forall i,j \in \left\{1,...,g \right\}  \]
We make the observation that in the second and third equation, $i$ and $j$ do not have to be distinct. \par 
\begin{defn}
	Under the flat metric of $\omega$, for a smooth loop $\gamma \subseteq X$ that does not pass through the zeroes of $\omega$ we define
	\[ Ind_\gamma(\omega) = \text{ \ degree of the Gauss map associated to  } \gamma . \] 
\end{defn}
\begin{ex}
	Let $p$ be a zero of order $2$ and $\gamma$ a smooth curve around $p$ as in the picture: \\ 
	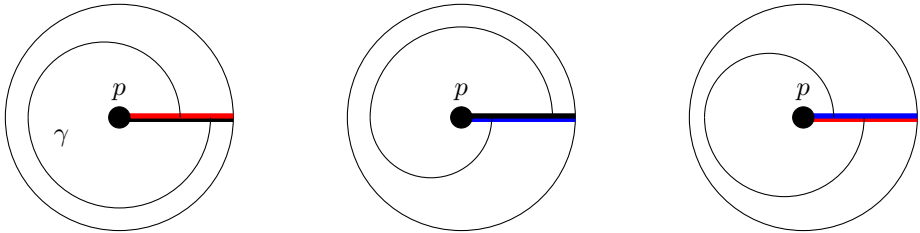
\begin{figure}[H]\centering
	\begin{tikzpicture}[baseline=(current bounding box.north)]
	
	% A clipped circle is drawn
	\begin{scope}
	
	\draw (1.5,0) arc(0:360:1.5cm );
	\draw[color = black,line width = 0.8mm] (0,-0.02) -- (1.5,-0.02);
	\draw[color = red, line width = 0.7mm] (0,0.02) -- (1.5,0.02
	);
	\node[mark size=4pt,color=black] at (0,0) {\pgfuseplotmark{*}};
	\draw (0.8,0) arc(0:180:1cm );
	\draw (-1.2,0) arc(180:360:1.2cm );
	\end{scope}
	%
	%%Labels for the vertices are typeset.
	\node[below left= 0.5mm of {(-0.5,0)}] {$\gamma$};
	\node[above= 1mm of {(0,0)}] {$p$};
	\node[above= 1mm of {(4.5,0)}] {$p$};
	\node[above= 1mm of {(9,0)}] {$p$};
	
	% A clipped circle is drawn
	\begin{scope}
	
	\draw (6,0) arc(0:360:1.5cm );
	\draw[color = blue, line width = 0.8mm] (4.5,-0.02) -- (6,-0.02);
	\draw[color = black, line width = 0.7mm] (4.5,0.02) -- (6,0.02);
	\node[mark size=4pt,color=black] at (4.5,0) {\pgfuseplotmark{*}};
	\draw (5.7,0) arc(0:180:1.2cm );
	\draw (3.3,0) arc(180:360:0.8cm );
	\end{scope}
	
	% A clipped circle is drawn
	\begin{scope}
	
	\draw (10.5,0) arc(0:360:1.5cm );
	\draw[color = red, line width = 0.8mm] (9,-0.02) -- (10.5,-0.02);
	\draw[color = blue, line width = 0.7mm] (9,0.02) -- (10.5,0.02);
	\node[mark size=4pt,color=black] at (9,0)
	{\pgfuseplotmark{*}};
	
	\draw (9.4,0) arc(0:180:0.85cm );
	\draw (7.7,0) arc(180:360:1.05cm );
	\end{scope}
	\end{tikzpicture} \caption{Smooth curve around a zero of order 2}
	\end{figure}
	A neighbourhood of the point $p$ looks like the above picture. The circles are glued along the slits, respecting the colour. Taking a curve $\gamma$ around the point $p$ as in the figure, we see that its Gauss map circles $S^1$ for three times and hence we have: 
	\[ Ind_\gamma(\omega) = 3. \]
\end{ex}
In this case, when $\omega$ has only even zeroes, we define the $Arf$ invariant.  
\begin{defn}
	Let $(X,\omega) \in \mathcal{H}(2k_1,...,2k_n)$. Then we define the $Arf$ invariant to be:
	\[ Arf(\omega) = \sum_{i=1}^g(Ind_{a_i}(\omega) + 1)(Ind_{b_i}(\omega) + 1) \pmod{2}.\]
\end{defn} \par
This is well-defined, independent of a choice of a basis respecting the intersection numbers we considered, as going across any of the even zero points of $\omega$ changes the index of $\gamma$ by an even number, thus preserving the parity. \\ 

In fact, we have even more. It was proven in \cite{abjoh80} that this invariant is actually equal to the parity of the spin structure. 
\begin{thm}
	In the above notations, we have: 
	\[ Arf(\omega) = \dim_\mathbb{C}H^0(X,k_1p_1+...+k_np_n) \pmod{2}. \] 
\end{thm}
\begin{ex}
	The most illustrative example we can take is that of $\mu = (2g-2)$ for $g\geq 4$, where we see both the hyperelliptic and spin structures come into play. We are now in a position to describe the points in each of the components. \begin{bulist} 
	\item $\mathcal{H}^{hyp}(2g-2) = \left\{(X,\omega) \mid X \text{ \ is hyperelliptic } \right\}$ 
	\item $\mathcal{H}^{even}(2g-2) = \left\{(X,\omega) \mid X \text{ \ is nonhyperelliptic and $\omega$ has even spin} \right\}$ 
	\item $\mathcal{H}^{odd}(2g-2) = \left\{(X,\omega) \mid X \text{ \ is nonhyperelliptic and $\omega$ has odd spin  } \right\}$ \end{bulist}
\end{ex}
\begin{remark}
	In the case $\mu = (g-1,g-1)$ and $g$ odd, it is not enough for the curve $X$ to be hyperelliptic to conclude that $(X,\omega) \in \mathcal{H}^{hyp}(g-1,g-1)$. \par 
	In fact, if $X$ is hyperelliptic and $\omega$ is a one-form with two zeroes of multiplicity $g-1$ at two Weierstrass points, then $(X,\omega)$ is contained in the odd spin component. \\
\end{remark}
\begin{remark}
	By the description of period coordinates, all the components of $\mathcal{H}(\mu)$ for any positive partition of $2g-2$ are smooth. In particular, it follows that all the connected components are irreducible. \\
\end{remark}
\begin{ex}  We take $g=1$. On a flat torus $X$ we have that 
	\[H^0(X,\mathcal{O}_X) = \left\{ \text{  constant functions \ } \right\} \cong \mathbb{C} \]
	\[\Rightarrow \dim_\mathbb{C}H^0(X,\mathcal{O}_X) = 1. \]
	Let us prove that this is equal to the $Arf$ invariant. To see this, consider the standard polygonal representation of the flat torus as in the figure. \\
	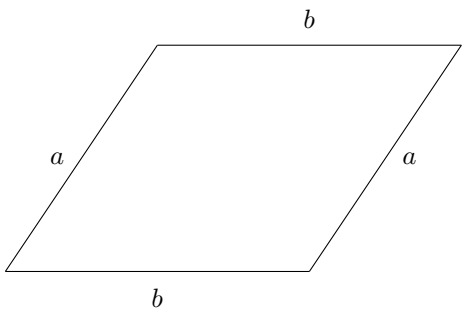
\begin{figure}[H] \centering
	\begin{tikzpicture}
	\begin{scope}
	\draw (0,0) -- (2,3);
	\draw(2,3) -- (6,3);
	\draw (0,0) -- (4,0);
	\draw (4,0) -- (6,3);
	\node[left = 1mm of {(1,1.5)}] {$a$};
	\node[right = 1mm of {(5,1.5)}] {$a$};
	\node[above= 1mm of {(4,3)}] {$b$};
	\node[below = 1mm of {(2,0)}] {$b$};
	\end{scope}
	\end{tikzpicture} \caption{Polygonal representation of flat torus}
	\end{figure}
	We can choose $a$ and $b$ as a symplectic basis satisfying the necessary relations. Since the normal vectors of $a$ and $b$ do not vary we conclude that 
	\[ Ind_a(\omega) = Ind_b(\omega) = 0 \]
	\[ \Rightarrow Arf(\omega) = (0+1)(0+1) = 1. \]
	It follows that for this example, the $Arf$ invariant and the parity of the spin structure coincide.
\end{ex}
\section{Teichm\"uller Dynamics} 
Take an element $(X,\omega)$ of $\mathcal{H}(\mu)$ and look at the polygonal representation of this translation surface. As it is seen in $\mathbb{R}^2$, we can act on it by the action of $GL_2^{+}(\mathbb{R})$ on $\mathbb{R}^2$. Of course, the action sends parallel edges to parallel edges and the glueing data stay the same. As a consequence, it follows that by acting on $(X,\omega)$ by a matrix in $GL_2^{+}(\mathbb{R})$ we get another translation surface in the same stratum $\mathcal{H}(\mu)$. \\ 

We get an action of $GL_2^{+}(\mathbb{R})$ on $\mathcal{H}(\mu)$. 
\begin{figure}[H] \centering

\begin{tikzpicture}
%%hereitende
% A clipped circle is drawn
\begin{scope}

\draw (0,0) -- (1,1.5);
\draw (1,1.5) -- (1,3);
\draw (1,3) -- (2.5,3);
\draw (2.5,3) -- (5,0.5);
\draw (5,0.5) -- (4,-1);
\draw (4,-1) -- (4,-2.5);
\draw (4,-2.5) -- (2.5,-2.5);
\draw (2.5,-2.5) -- (0,0);
%andere

\draw (8,-2.5) -- (9,0);
\draw (9,0) -- (9,1.5);
\draw (9,1.5) -- (10.5,3);
\draw (10.5,3) -- (13,3);
\draw (13,3) -- (12,0.5);
\draw (12,0.5) -- (12,-1);
\draw (12,-1) -- (10.5,-2.5);
\draw (10.5,-2.5) -- (8,-2.5);
%arrow
\draw [decoration={markings,mark=at position 1 with
	{\arrow[scale=3,>=stealth]{>}}},postaction={decorate}] (6,0) -- (7.5,0);
%notations
\node[right = 1mm of {(5.5,0.7)}] {$A = \begin{bmatrix} * & * \\ * & * \end{bmatrix}$};
\node[left = 1mm of {(0.5,0.8)}] {$v_1$};
\node[right = 1mm of {(7.7,-1.1)}] {$u_1$};
\node[right = 1mm of {(0.3,2.1)}] {$v_2$};
\node[right = 1mm of {(8.3,0.5)}] {$u_2$};
\node[left = 1mm of {(2,3.2)}] {$v_3$};
\node[right = 1mm of {(9.1,2.4)}] {$u_3$};
\node[right = 1mm of {(3.7,2)}] {$v_4$};
\node[right = 1mm of {(11.3,3.2)}] {$u_4$};
%again
\node[left = 1mm of {(5.2,-0.3)}] {$v_1$};
\node[right = 1mm of {(12.4,1.6)}] {$u_1$};
\node[right = 1mm of {(4,-1.7)}] {$v_2$};
\node[right = 1mm of {(12,-0.3)}] {$u_2$};
\node[left = 1mm of {(3.6,-2.8)}] {$v_3$};
\node[right = 1mm of {(11.1,-1.9)}] {$u_3$};
\node[right = 1mm of {(0.7,-1.5)}] {$v_4$};
\node[right = 1mm of {(9,-2.7)}] {$u_4$};
\end{scope} 
\end{tikzpicture} \caption{Action of a matrix $A$ on a translation surface}
\end{figure}
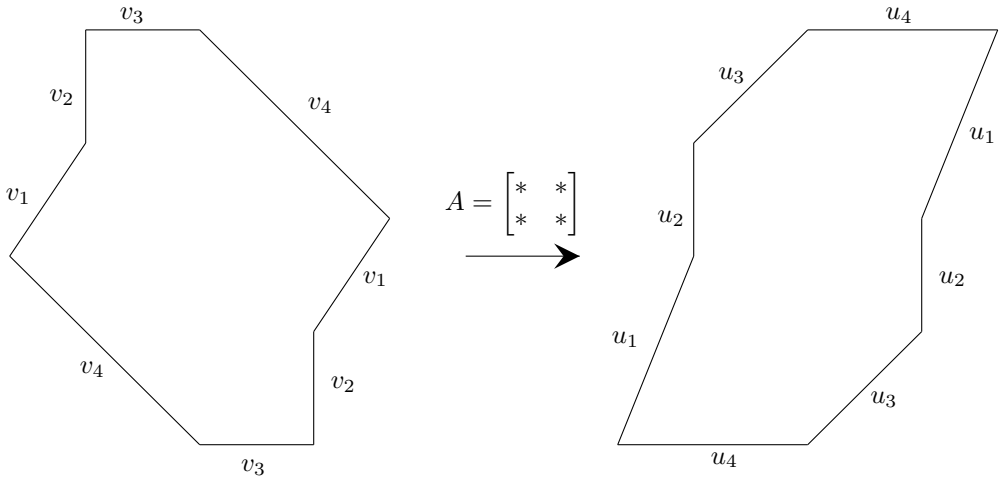
\par 
Here we have 
\[ u_i = Av_i \text{ \ } \forall i  \]
\\
Take a translation surface $(X,\omega) \in \mathcal{H}(\mu)$, given by a polygon and look at its orbit with regard to $GL_2^{+}(\mathbb{R})$ in the stratum. We would be interested in describing this orbit. The next theorem is a consequence of results in \cite{abmas82} and \cite{abvee82}, addressing the question of the closure of such an orbit under the analytic topology of period coordinates. However, the techniques used for proving this result are not algebraic. \\
\begin{thm}(Masur, Veech, '80s) 
	Take $(X,\omega)$ a generic element of $\mathcal{H}(\mu)$. Then the closure of its orbit with regard to $GL_2^{+}(\mathbb{R})$ is a connected component of the stratum $\mathcal{H}(\mu)$.
\end{thm}
However, for special $(X,\omega)$ the orbit closure can be a proper subset of $\mathcal{H}(\mu)$. The classification of such proper subsets is a central theme in Teichm\"uller Dynamics. A fundamental result in answering this question was achieved recently in \cite{abem18} and \cite{abemm15}. In understanding such orbit closures, the period coordinates we used to motivate the dimension of $\mathcal{H}(\mu)$ will play an essential role, as we will now see. \\

\begin{thm}(Eskin-Mirzakhani-Mohammadi, 2013) \label{ab:emm}
	Any orbit closure is locally linear in $\mathcal{H}(\mu)$, namely it is locally cut out by real and homogeneous linear equations of period coordinates of $\mathcal{H}(\mu)$.
\end{thm}
\par 
This theorem implies that orbit closures are in fact analytic submanifolds of $\mathcal{H}(\mu)$. We would be interested to say something even more: that these submanifolds also have an algebraic structure, and hence replacing "analytic submanifold" by "algebraic subvariety". This was achieved shortly after in a companion theorem, appearing in \cite{abfil16}. \\

\begin{thm}(Filip) 
	Any orbit closure is locally cut out by algebraic and homogeneous linear equations of period coordinates of $\mathcal{H}(\mu)$. In particular, the orbit closures in $\mathcal{H}(\mu)$ are algebraic subvarieties defined over number fields.
\end{thm}
\subsection{Teichm\"uller curves} \par 
We will start with the definition and then provide some information on the importance of Teichm\"uller curves. \par
\begin{defn} \par 
	Consider the projection morphism 
	\[\mathcal{H}(\mu) \rightarrow \mathcal{M}_g. \] \par 
	Take an element $(X,\omega)$ and consider the action of $GL_2^{+}(\mathbb{R})$ on it. One observation we should make is that, by acting with an element of $SO(2)$ on a polygonal representation, we are not changing the underlying Riemann surface $X$. It follows that the projection of the $GL_2^{+}(\mathbb{R})$ orbit of $(X,\omega)$ factors through the upper half plane $\mathbb{H}$. \par 
	The induced map (or its image) $\mathbb{H} \rightarrow \mathcal{M}_g $ is called a Teichm\"uller disk. If the image in $\mathcal{M}_g$ is a complex algebraic curve, then it is called a Teichm\"uller curve.  \\
\end{defn}
 
An immediate consequence of this definition is that, for a Teichm\"uller curve,  the associated orbit closure in $\mathcal{H}(\mu)$ has minimal possible dimension. In fact, these curves satisfy many other fascinating properties which we will now mention, to provide the reader with a glimpse of their importance. \begin{bulist}
\item Under a certain metric, Teichm\"uller curves are local isometries from a curve to $\mathcal{M}_g$.(\cite{absw04},\cite{abvee95}) 
\item The union of all Teichm\"uller curves is a dense subset of $\mathcal{M}_g$ (\cite{abeo01},\cite{abche11}) 
\item Teichm\"uller curves are rigid. \cite{abmcm09} 
\item Every curve over a number field is birational to a Teichm\"uller curve. \cite{abem12}  
\item Teichm\"uller curves are not complete curves in $\mathcal{M}_g$ and their closures in the Deligne-Mumford compactification $\overline{\mathcal{M}}_g$ do not intersect any boundary divisors except for $\Delta_0$. \end{bulist}
All these results, and some more are mentioned in \cite{abche17}. Apparently from these results we can see the importance of Teichm\"uller curves. \\

Let us proceed and provide an example of the instance where Teichm\"uller curves appear. 
\subsection{Branched cover construction for (special) Teichm\"uller curves}
\ \par
Consider a holomorphic degree $d$ branched cover
\[ \pi\colon X \rightarrow E\] 
from a genus $g$ Riemann surface $X$ to the square torus $E$, satisfying the following properties: \begin{bulist} 
\item $\pi$ has a unique branch point at a point $q \in E$; 
\item $\pi$ has ramification points $p_1,...,p_n$ over $q$; 
\item each $p_i$ has ramification order $m_i$. \end{bulist} 
\begin{remark}
	From this definition we can obtain a holomorphic one-form on $X$ as follows:
	
	\[ \omega = \pi^{*}dz. \] 
	Its zeroes and multiplicities are easy to understand using the properties we just described. At $p_i$ the map $\pi$ is given by: 
	\[ u \longmapsto z = u^{m_i+1}. \]
	Hence we have
	\[ \pi^{*}dz = d(u^{m_i+1}) \sim u^{m_i}du. \]
	It follows that $\omega$ has a zero of order $m_i$ at $p_i$ and hence $(X,\omega) \in \mathcal{H}(\mu)$ for $\mu=(m_1,...,m_n)$.
\end{remark}
\begin{defn}
	An element $(X,\omega)$ satisfying the properties outlined above is called a square-tiled surface. \\
\end{defn} 
\begin{ex}
	Take the case of genus $g = 2$ and the degree of the map $\pi$ is $5$. As before, $E$ is the square torus. We will provide, geometrically an example of a square-tiled surface. 
	\begin{figure}[H] \centering
	\begin{tikzpicture}
	\begin{scope}
	\draw (0,0) -- (0,3);
	\draw[dashed] (0,1.5) -- (1.5,1.5);
	\draw[dashed] (1.5,0) -- (1.5,1.5);
	\draw[dashed] (3,0) -- (3,1.5);
	\draw[dashed] (4.5,0) -- (4.5,1.5);
	\draw (0,0) -- (6,0);
	\draw (0,3) -- (1.5,3);
	\draw (1.5,3) -- (1.5,1.5);
	\draw (1.5,1.5) -- (6,1.5);
	\draw (6,0) -- (6,1.5);
	\draw (11,0) -- (11,1.5);
	\draw (11,0) -- (12.5,0);
	\draw (12.5,1.5) -- (12.5,0);
	\draw (11,1.5) -- (12.5,1.5);
	\draw [decoration={markings,mark=at position 1 with
		{\arrow[scale=3,>=stealth]{>}}},postaction={decorate}] (7.5,0.75) -- (9.5,0.75);
	\node[mark size=2pt,color=black] at (0,1.5) {\pgfuseplotmark{*}};
	\node[mark size=2pt,color=black] at (1.5,0) {\pgfuseplotmark{*}};
	\node[mark size=2pt,color=black] at (0,0) {\pgfuseplotmark{*}};
	\node[mark size=2pt,color=black] at (0,3) {\pgfuseplotmark{*}};
	\node[mark size=2pt,color=black] at (6,0) {\pgfuseplotmark{*}};
	\node[mark size=2pt,color=black] at (6,1.5) {\pgfuseplotmark{*}};
	\node[mark size=2pt,color=black] at (1.5,1.5) {\pgfuseplotmark{*}};
	\node[mark size=2pt,color=black] at (1.5,3) {\pgfuseplotmark{*}};
	\node[below = 1mm of {(0.75,0)}] {$v_1$};
	\node[above = 1mm of {(0.75,3)}] {$v_1$};
	\node[right = 1mm of {(1.5,2.25)}] {$v_2$};
	\node[left = 1mm of {(0,2.25)}] {$v_2$};
	\node[below = 1mm of {(3.75,0)}] {$v_3$};
	\node[above = 1mm of {(3.75,1.5)}] {$v_3$};
	\node[right = 1mm of {(6,0.75)}] {$v_4$};
	\node[left = 1mm of {(0,0.75)}] {$v_4$};
	\node[left = 1mm of {(11,0.75)}] {$a$};
	\node[right = 1mm of {(12.5,0.75)}] {$a$};
	\node[above = 1mm of {(11.75,1.5)}] {$b$};
	\node[below = 1mm of {(11.75,0)}] {$b$};
	\end{scope}
	\end{tikzpicture}  \caption{Example of square-tiled surface} 
	\end{figure}
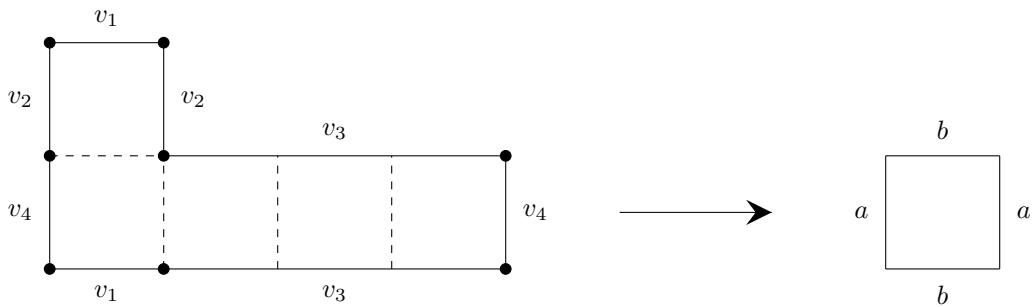
	Here we used the bullet points to make the separation between different vectors $v_i$ and $v_j$ clear. \\
	
	We easily see that the only ramification point is $q$  (corresponding to the vertices of the square) and we can check that $\pi^{*}dz$ has a double zero at $p$ (corresponding to the vertices of the polygon), i.e. angle $6\pi$ at $p$. \\
\end{ex}
This polygonal representation of $X$ is very useful in understanding its orbit. In our case, the group $GL_2^{+}(\mathbb{R})$ acts on $X$ just by changing the squares by parallelograms. 
\begin{prop}
	The orbit of such square-tiled surfaces modulo the scaling factor $\mathbb{C}^{*}$ is the complex one-dimensional Hurwitz space of degree $d$, genus $g$ covers of elliptic curves with a unique branch point of ramification type $\mu$. \\
\end{prop}
\begin{remark}
	Such one-dimensional Hurwitz spaces give infinitely many Teichm\"uller curves in $\mathcal{H}(\mu)$, as $d$ can be arbitrarily large.
\end{remark} 
The equations in the period coordinates for the orbit in $\mathcal{H}(2)$ are the following: 
\[v_4-v_2=0, \] 
\[3v_1-v_3 = 0. \] \par
Thus the validity of Theorem \ref{ab:emm} is checked in this particular case. \\

We will finish these notes by mentioning some possible generalizations of the concepts presented here. Of course, it is natural to consider also the strata of meromorphic differentials and higher order differentials. We would be interested in understanding their connected components, cycle classes in moduli spaces and compactification. The question of compactification is addressed in \cite{abbcg18, abbcg19} while their cycle classes were studied in \cite{absau19}. Moreover, F. Janda, R. Pandharipande, A. Pixton, and D. Zvonkine provide a conjectural description of the cycle classes in the case of differential one-forms in Appendix A of \cite{abfp18}. This conjectural description was generalized to the case of $k$-differentials in \cite{absch18}. \\ 

We will now proceed to present the exercise sheet that accompanies these lectures and also, provide solutions for some of the problems.

\newpage \ \\ \ \\ \ \\ \ \\ 

\begin{center}
	\large{Moduli of Differentials and Teichm\"uller Dynamics}
\end{center}

\begin{center}
	\large{Exercises}
\end{center}
\ \\  \ \\ \ \\
\textbf{Exercise 1.} Describe the zeroes (and poles) of the differentials represented by the following flat surfaces: \par
(1) A decagon of type $a+b+c+d+e = e+b+a+d+c$. \par
(2) A $2n$-gon of type $v_1+...+v_n=v_n+...+v_1$. \par
(3) A big flat torus minus a small flat torus. \par
(4) The Euclidean plane with a point at infinity minus a flat torus. \par
(5) A pillow case. \par 
(6) The surface of a cube. \par
Remark: (4), (5), (6) are generalizations of holomorphic one-forms. \\[2mm]
\textbf{Exercise 2.} Prove that $\mathcal{H}(4)$ has exactly two connected components. \\[2mm]
\textbf{Exercise 3.} Draw a translation surface in $\mathcal{H}(2)$ and compute the Arf invariant. \\[2mm]
\textbf{Exercise 4.} A variety is called unirational if it can be dominated by a projective space. Prove that each of the following strata is unirational. \par 
(1) $\mathcal{H}^{hyp}(2g-2)$. \par 
(2) All strata (components) in genus 3. \\[2mm]
\textbf{Exercise 5.} Prove that for any translation surface in $\mathcal{H}(3,1)$, the underlying Riemann surface is not hyperelliptic. \\[2mm]
\textbf{Exercise 6.} Suppose $(X,\omega) \in \mathcal{H}(2,1,1)$ with $(\omega)_0 = 2p_1+p_2+p_3$. Prove that $h^0(X,p_1+p_2) = 1.$ \\[2mm]
\textbf{Exercise 7.} Prove that the projection of any $GL_2^{+}(\mathbb{R})$-orbit to the moduli space of genus $g$ curves $\mathcal{M}_g$ factors through the upper-half plane $\mathbb{H}.$ \\[2mm]
\textbf{Exercise 8.} Show that $(X,\omega) \in \mathcal{H}(\mu)$ corresponds to a square-tiled surface if and only if all period coordinates of $(X,\omega)$ belong to $\mathbb{Z} \oplus \mathbb{Z}i$, namely, if and only if $(X,\omega)$ is an integral point in $\mathcal{H}(\mu)$ under the period coordinates. \\[2mm]
\textbf{Exercise 9.} Let $\Delta = \cup_{i=0}^{[g/2]}\Delta_i$ be the total boundary of the Deligne-Mumford compactification $\overline{\mathcal{M}}_g$, where a general point in the boundary component $\Delta_i$ parametrizes a nodal union of a genus $i$ curve and a genus $g-i$ curve for $i>0$ and a general point in $\Delta_0$ parametrizes an irreducible nodal curve of geometric genus $g-1$. Prove that for a Teichm\"uller curve in $\mathcal{M}_g$ generated by a square-tiled surface, its closure in $\overline{\mathcal{M}}_g$ does not intersect $\Delta_i$ for any $i>0$. \\[2mm]
\textbf{Exercise 10.} Suppose a family of translation surfaces in $\mathcal{H}(2)$ degenerate to two elliptic curves $E$ and $E'$ union at a node $q$. Moreover suppose the limit of the double zeros is a point $p \in E \setminus q$. Prove that $2p \sim 2q$ in $E$.  \\[2mm]
\textbf{Exercise 11.} Show that in Exercise 1, part (2), the underlying Riemann surface is hyperelliptic. Moreover, find all of its Weierstrass points. \\[2mm]
\textbf{Exercise 12.} Let $X$ be a hyperelliptic Riemann surface defined by the equation
\[ x^2=(z-a_1)...(z-a_{2g+2}) \]
(extending to $\infty$), where $a_1,...,a_{2g+2}$ are fixed distinct points in $\mathbb{C}$ and the double cover $X\rightarrow \mathbb{P}^1$ is given by $(x,z)\rightarrow z$. Prove that the following one-forms are holomorphic (hence they form a basis of the space of holomorphic one-forms on $X$):
\[ \frac{dz}{x}, \frac{zdz}{x},...,\frac{z^{g-1}dz}{x}. \] \\[2mm]
\textbf{Exercise 13.} Draw two translation surfaces in $\mathcal{H}(4)$ such that their Arf invariants are different. 

\section*{Some solutions} 
\textbf{Exercise 1.}(2)
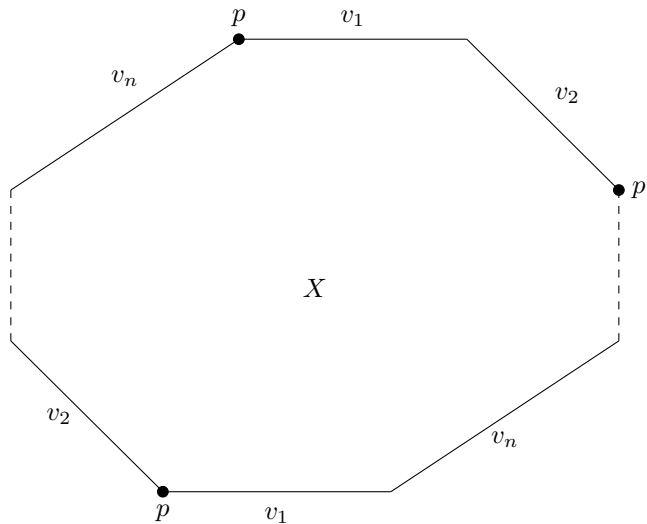
\begin{figure}[H] \centering
\begin{tikzpicture}
\begin{scope}
\draw (0,0) -- (3,2);
\draw (3,2) -- (6,2);
\draw (6,2) -- (8,0);
\draw[dashed] (0,0) -- (0,-2);
\draw[dashed] (8,0) -- (8,-2);
\draw (0,-2) -- (2,-4);
\draw (2,-4) -- (5,-4);
\draw (5,-4) -- (8,-2);

\node[above = 0.5mm of {(3,2)}] {$p$};
\node[right = 0.5mm of {(8,0)}] {$p$};
\node[below = 0.5mm of {(2,-4)}] {$p$};
\node[below = 0.5mm of {(4,-1)}] {$X$};
\node[above = 0.5mm of {(1.5,1.2)}] {$v_n$};
\node[above = 0.5mm of {(4.5,2)}] {$v_1$};
\node[above right = 0.5mm of {(7,1)}] {$v_2$};
\node[below = 0.8mm of {(3.5,-4)}] {$v_1$};
\node[left = 0.8mm of {(1,-3)}] {$v_2$};
\node[below = 0.8mm of {(6.5,-3)}] {$v_n$};

\node[mark size=2pt,color=black] at (3,2) {\pgfuseplotmark{*}};
\node[mark size=2pt,color=black] at (2,-4) {\pgfuseplotmark{*}};
\node[mark size=2pt,color=black] at (8,0) {\pgfuseplotmark{*}};
\end{scope}
\end{tikzpicture} \caption{Translation surfaces defined by the corresponding vectors} 
\end{figure}
We distinguish two cases, depending on the parity of $n$. \begin{numlist} 
\item The number $n$ even. 
\par 
In this case, we can see by chasing the glueings that all vertices identify with each other: \par
If we denote by $p$ the starting point of the vector $v_1$, we see that it is the same as the endpoint of the vector $v_2$ which is the same as the starting point of $v_3$ and so on. Inductively we see that the vertices correspond to each other by simple combinatorics. \par 
It follows that the angle around $p$ is $(2n-2)\pi$. \par 
Let us compute the genus of the curve using the Euler characteristic: \begin{abclist} 
\item one vertex $p$, 
\item $n$ edges $v_1,...,v_n$, 
\item one face $X$. \end{abclist}
\[ \Rightarrow 2-2g = 2-n \Rightarrow g = \frac{n}{2}. \]
The order of multiplicity at $p$ of the corresponding one-form is $n-2$ and we see that the translation surface is in the stratum $\mathcal{H}(n-2)$. \par 
\item The number $n$ is odd. \par 
In this case, we can check that the vertices are identified alternatingly with two points $p$ and $q$. The angle around each of them is $(n-1)\pi$. 
A similar genus computation yields that $g(X) = \frac{n-1}{2}$. \par 
We see that in this case the corresponding translation surface is in $\mathcal{H}(\frac{n-3}{2}, \frac{n-3}{2})$. \end{numlist}
\textbf{Exercise 1.}($3$) 
\begin{figure}[H] \centering
\begin{tikzpicture}
\begin{scope}
\draw (0,0) -- (2,6);
\draw (0,0) -- (8,0);
\draw (8,0) -- (10,6);
\draw (10,6) -- (2,6);
\draw (4,2) -- (4,3);
\draw (4,2) -- (5,2);
\draw (5,3) -- (5,2);
\draw (5,3) -- (4,3);
\node[left = 0.5mm of {(1,3)}] {$a$};
\node[right = 0.5mm of {(9,3)}] {$a$};
\node[below = 0.5mm of {(4,0)}] {$b$};
\node[above = 0.5mm of {(6,6)}] {$b$};
\node[above = 0.5mm of {(4.5,3)}] {$c$};
\node[below = 0.5mm of {(4.5,2)}] {$c$};
\node[left = 0.5mm of {(4,2.5)}] {$d$};
\node[right = 0.5mm of {(5,2.5)}] {$d$};
\end{scope}
\end{tikzpicture} \caption{Translation surface given by a flat torus minus a smaller flat torus} 
\end{figure}
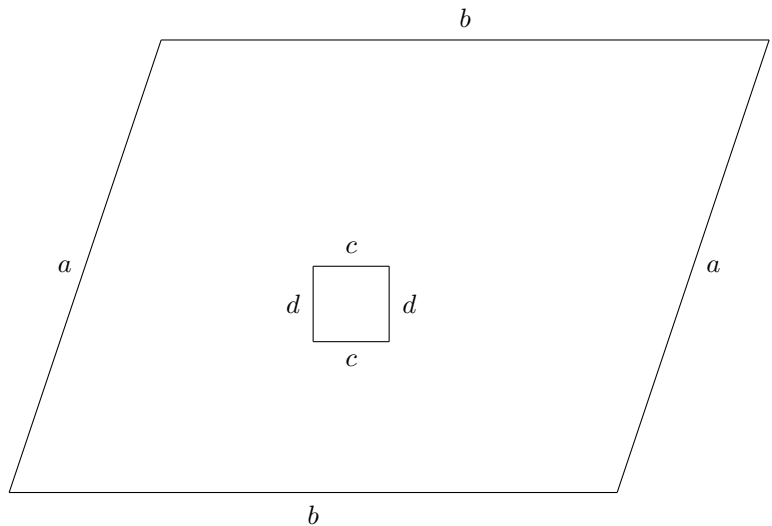 
By identifying the corresponding edges in the above picture we see that the vertices of the big torus all identify with each other and the same is true for the small torus. We denote the corresponding two points by $p$ and $q$. \\

 The angle around these two points are $2\pi$ and $6\pi$ (as we count the exterior angles of the small parallelogram). As no other point except $q$ is special, it follows that the corresponding translation surface lies in $\mathcal{H}(2)$. This also implies that the genus is 2. Indeed, by choosing a triangulation of the surface we can compute using the Euler characteristic that the genus is 2. \\
 
\textbf{Exercise 1.}(4) \par 
The points of the torus get identified with one another and the angle around the point is the sum of the exterior angles, hence it is equal to $6\pi$. On the other hand, we know that when it is extended, the one-form $dz$ has a double pole at infinity, so we see that the translation surface is in $\mathcal{H}(2,-2)$. Using half-lines to triangulate our surface we check by the Euler characteristic that the genus of our surface is indeed 1. \par 
This example is interesting from the following two points of view: \begin{bulist}
\item This can be seen as the limit case of the Riemann surface in part ($3$) by taking the big torus to infinity. However, there seems to be a problem with this analogy: the limit is genus $1$ while we started with genus $2$. In fact, there is a hidden component: the one we obtain by shrinking the small torus to a point instead of expanding the big torus. This hidden component corresponds to an element in $\mathcal{H}(0)$. In fact, the limit curve is the nodal curve obtained by glueing the two genus 1 curves described along the infinity point and $p$ (the point to which the small torus degenerates). We are satisfied as now the genus is preserved going to the limit case.
\item As seen above, meromorphic differentials naturally arise in the boundary of a holomorphic stratum $\mathcal{H}(\mu)$, motivating in part the interest in also studying the meromorphic strata. \end{bulist}
\textbf{Exercise 1.}(5) \par 
Let us first clarify what we mean by a "pillow case" by providing a geometric picture explaining the glueings. 
\begin{figure}[H] \centering
\begin{tikzpicture}
\begin{scope}
\draw[->] (0,0) -- (3,0);
\draw[->] (0,0) -- (0,3);
\draw[->] (6,0) -- (3,0); 
\draw[->] (6,0) -- (6,3);
\draw[->] (0,3) -- (3,3);
\draw[->] (6,3) -- (3,3); 
\node[above = 0.5mm of {(1.5,3)}] {$b$};
\node[above = 0.5mm of {(4.5,3)}] {$b$};
\node[left = 0.5mm of {(0,1.5)}] {$a$};
\node[right = 0.5mm of {(6,1.5)}] {$a$};
\node[below = 0.5mm of {(1.5,0)}] {$c$};
\node[below = 0.5mm of {(4.5,0)}] {$c$};
\node[below = 0.5mm of {(4.5,0)}] {$c$};
\node[above = 0.5mm of {(3,3)}] {$r$};
\node[above left = 0.5mm of {(0,3)}] {$p$};
\node[below left = 0.5mm of {(0,0)}] {$q$};
\node[above right = 0.5mm of {(6,3)}] {$p$};
\node[below right = 0.5mm of {(6,0)}] {$q$};
\node[below = 0.5mm of {(3,0)}] {$s$};
\end{scope}
\end{tikzpicture} \caption{Pillow case} 
\end{figure}
The segments here are glued together respecting the orientation of the vector. In particular, as the glueing are not determined by translations, this geometric picture does not correspond to a one-form. However, they are determined by either translation or rotation of $180$ degree so we have a distinguished quadratic differential on the associated Riemann surface. \\

 We see that the four special points are $p,q,r,s$ and at each of the points we have a pole of the quadratic differential. In particular, the half-translation surface in this case has four simple poles. This would imply that the genus of the Riemann surface is 0. \\
 
Let us check this by using the Euler characteristic: 
\begin{figure}[H]\centering
\begin{tikzpicture}
\begin{scope}
\draw[->] (0,0) -- (3,0);
\draw[->] (0,0) -- (0,3);
\draw[->] (6,0) -- (3,0); 
\draw[->] (6,0) -- (6,3);
\draw[->] (0,3) -- (3,3);
\draw[->] (6,3) -- (3,3); 
\draw[dashed] (0,0) -- (3,3);
\draw[dashed] (3,0) -- (3,3);
\draw[dashed] (3,0) -- (6,3);
\node[above = 0.5mm of {(1.5,3)}] {$b$};
\node[above = 0.5mm of {(4.5,3)}] {$b$};
\node[left = 0.5mm of {(0,1.5)}] {$a$};
\node[right = 0.5mm of {(6,1.5)}] {$a$};
\node[below = 0.5mm of {(1.5,0)}] {$c$};
\node[below = 0.5mm of {(4.5,0)}] {$c$};
\node[below = 0.5mm of {(4.5,0)}] {$c$};
\node[above = 0.5mm of {(3,3)}] {$r$};
\node[above left = 0.5mm of {(0,3)}] {$p$};
\node[below left = 0.5mm of {(0,0)}] {$q$};
\node[above right = 0.5mm of {(6,3)}] {$p$};
\node[below right = 0.5mm of {(6,0)}] {$q$};
\node[below = 0.5mm of {(3,0)}] {$s$};
\end{scope} 
\end{tikzpicture} \caption{Triangulation of the surface}
\end{figure}
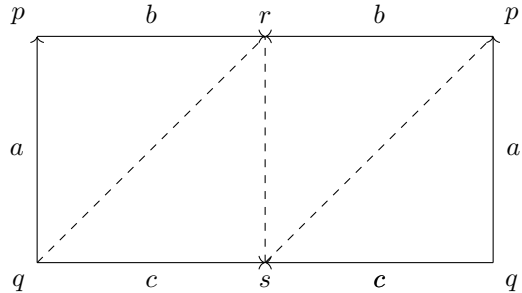
In this triangulation we have: \begin{bulist}
\item 4 vertices: $p,q,r,s$ 
\item 6 edges: $a,b,c$ and the three dashed segments 
\item 4 faces: determined by the four triangles in the picture. \end{bulist}
\par 
It follows that the genus is 0, as expected. \\

\textbf{Exercise 2.}  
We have seen in the lectures that $\mathcal{H}^{hyp}(4)$ is connected. Hence we sketch that there is only one other connected component. The key observation is that if $X$ is a genus 3, nonhyperelliptic curve, it can be embedded in $\mathbb{P}^2$ using the global sections of the canonical bundle. Hence $X$ has a canonical embedding as a plane quartic. Denote also by $X$ the image in $\mathbb{P}^2$. Then each line section of $X$ in $\mathbb{P}^2$ corresponds to the zero divisor of a holomorphic one-form. \\ 
 
For our purpose we want to have a line that intersects $X$ in a point of multiplicity 4. We consider the following incidence variety: 
\[ \Sigma = \left\{(X,L,p) \mid \text{the line $L$ has contact order 4 at $p$ with the quartic $X$}\right\}. \] \par
Consider the flag variety $F$ of pairs $(p,L)$ such that the point $p$ is contained in the line $L$. \\ 

The map $\Sigma \rightarrow F $ has fibers given by linear subspaces of quartics and $F$ is irreducible. It follows that $\Sigma$ is irreducible. In particular, it follows that the image of $\Sigma$ in $\mathcal{M}_{3,1}$ is irreducible. The image consists of points $(X,p)$ such that $X$ is not hyperelliptic and 
\[\mathcal{O}_X(4p) \cong K_X. \]
There is a projection with fibers $\mathbb{C}^{*}$ from the corresponding nonhyperelliptic locus in $\mathcal{H}(4)$ to this irreducible image in $\mathcal{M}_{3,1}$ and hence the conclusion follows. \\ \ \\
\textbf{Exercise 5} Take $(X,\omega)$ in $\mathcal{H}(3,1)$ and assume that $X$ is hyperelliptic. We will argue by contradiction. \par 
It is known that in this case $X$ has a unique $g^1_2$ and that $K_X = 2g^1_2$. Consider now the points $p$ and $q$ on $X$ such that: 
\[ (\omega)_0= 3p+q. \]
By what we just said, it follows that $p$ must be a Weierstrass point and moreover that $p$ and $q$ are conjugate. It then follows that $p = q$, because a Weierstrass point is conjugate to itself, which contradicts to the fact that they are distinct by definition. \\

In conclusion it follows that the underlying Riemann surface of a translation surface in $\mathcal{H}(3,1)$ cannot be hyperelliptic. \\ \ \\
\textbf{Exercise 6.} We will argue by contradiction. \par

Assume there exists $(X,\omega) \in \mathcal{H}(2,1,1)$ such that 
\[ h^0(X,p_1+p_2) = 2. \]
We preserve here the notations in the statement of the exercise. Consider 
\[ s\in H^0(X,p_1+p_2) \]
a nonconstant global section. \\

Then $s$ determines a morphism from $X$ to $\mathbb{P}^1$ with fiber over 0 consisting of the points $p_1$ and $p_2$, both with multiplicity 1. It follows that this morphism is a degree 2 map to $\mathbb{P}^1$ and $X$ is hyperelliptic. By this description it follows that $p_1$ and $p_2$ are conjugate. Using the same approach as in Exercise 5, we also conclude that $p_1$ and $p_3$ are conjugate. 
\[ \Rightarrow p_2=p_3 \]
but they are different by definition. It follows that our assumption is wrong and that 
\[ h^0(X,p_1+p_2) = 1.\]
\\
\textbf{Exercise 7.} We have that 
\[GL_2^{+}(\mathbb{R})/\mathbb{R}^{+}SO(2) \cong \mathbb{H}. \] \par
The observation is that acting by an element of $\mathbb{R}^{+}SO(2)$ on a polygon does not change the underlying Riemann surface. But this is clear as rotations and homotheties preserve the complex structure. The conclusion follows. \\ \ \\
\textbf{Exercise 8.} A variant of this exercise can be find as Lemma 3.1 in \cite{abeo01}. Take an element $(X,\omega) \in \mathcal{H}(\mu)$. Recall first that we have locally on $\mathcal{H}(\mu)$ the period coordinates, given by a basis of the relative homology $H_1(X,p_1,...,p_n;\mathbb{Z})$ where $p_1,...,p_n$ are the zeroes of $\omega$. \\

$"\Rightarrow"$ Let $(X,\omega)$ a square-tiled surface. Then we know that there exists a map 
\[ \pi \colon X \rightarrow E \cong \mathbb{C}/(\mathbb{Z}\oplus \mathbb{Z}i)\ \]
such that it has a unique branch point $q \in E$ and is ramified only over $p_1,...,p_n$ with ramification order given by $\mu$ and $\omega = \pi^{*}dz$. \\

We prove that all period coordinates of $\omega$ are in $\mathbb{Z}\oplus \mathbb{Z}i$. We distinguish two cases: \begin{numlist}
\item $\gamma$ comes from the absolute homology; 
\item $\gamma$ comes from the relative part. \end{numlist} 
Consider the case 1: $\gamma$ a closed curve. We know from the projection formula that: 
\[ \int\limits_\gamma \pi^{*}dz = \int\limits_{\pi_{*}\gamma}dz. \] \par 
Since $\gamma$ is closed, it follows that $\pi_{*}\gamma$ is closed, and
\[\pi_{*}\gamma \in H_1(E;\mathbb{Z}), \]
hence the period of $\omega$ at $\gamma$ is an integral sum of $\alpha$ and $\beta$, the standard basis of $H_1(E;\mathbb{Z})$. 
\[\Rightarrow \int\limits_{\pi_{*}\gamma}dz \in \mathbb{Z}\oplus\mathbb{Z}i. \] \par 
Case 2 can be similarly treated: 
\[ \int\limits_\gamma \pi^{*}dz = \int\limits_{\pi_{*}\gamma}dz. \] \par 
The curve $\gamma$ has endpoints $p_1$ and $p_i$ which both project to $q$ in $E$. In particular, the curve $\pi_{*}\gamma$ becomes closed and the same reasoning as before gives 
\[ \int\limits_{\pi_{*}\gamma}dz \in \mathbb{Z}\oplus\mathbb{Z}i. \] \\

$"\Leftarrow"$ A standard way to obtain local charts on $X$ is to consider integrals: 
\[ x\in X \mapsto \int\limits_\gamma \omega, \] 
where $\gamma$ is a path from a fixed point $p$ to $x$. However, this map is not well-defined, as it depends on the choice of $\gamma$. \\

Since all period coordinates are in $\mathbb{Z}\oplus\mathbb{Z}i$, to fix the issue, we modify the above map as follows: 
\[ x\mapsto \int\limits_\gamma \omega \text{\ (mod }\mathbb{Z}\oplus\mathbb{Z}i). \] \par
We claim that this is independent of the choice of the path $\gamma$ from $p$ to $x$. \par
Take $\gamma_1$ and $\gamma_2$ two such paths. It follows that $\gamma_1 \cdot \gamma_2^{-1}$ is a closed path and can be written as an integral combination of elements in the absolute homology.
\[\Rightarrow \int\limits_{\gamma_1\cdot \gamma_2^{-1}}\omega \in \mathbb{Z}\oplus\mathbb{Z}i\]
\[\Rightarrow \int\limits_{\gamma_1}\omega= \int\limits_{\gamma_2}\omega \text{\ (mod }\mathbb{Z}\oplus\mathbb{Z}i). \] 
Therefore, this induces a map from $X$ to the unique square torus $E$. \\ 

The points $p_1,p_2,...,p_n$ are all mapped to the same point $q$ because of the conditions on the relative homology. \\

The last thing that we will need to see is that the map defined in this way has order of ramification $m_i$ at $p_i$. \\ 

Locally at $p_i$, $\omega$ looks like $z^{m_i}dz$. It then follows that its integral looks like $z^{m_i+1}$ near the point $p_i$ and hence the ramification order is the one expected. \par 
It is easy to see either by looking at the one-form $\omega$ or by the Riemann-Hurwitz Formula, that the map has no other ramification points. \\ \ \\
\textbf{Exercise 12.} We make the observation that locally near the points where $x=0$ we have that $x$ is a local coordinate while $z$ is not. By using the relation between $dx$ and $dz$ that we deduce from the equation, we see that the one-forms in the exercise are holomorphic at the points $(0,a_i)$. We would just need to check that when extended to $\infty$, these one-forms remain holomorphic. \\ 

Denote $p(z) = (z-a_1)(z-a_2)...(z-a_{2g+2})$. We take the following polynomial: 
\[ k(t) = t^{2g+2}p(1/t). \] 
Take the surface given in $\mathbb{C}^2$ by the equation: 
\[y^2 = k(t). \]
We are glueing the two surfaces away from $x=0$ and $y=0$ by: 
\[ t = \frac{1}{z}\text{ \ and \ } y = \frac{x}{z^{g+1}}. \]
By these glueings, we obtain a compactification of the Riemann surface described by the first equation and moreover, by writing the one-forms in terms of $y$ and $t$, we see that they are holomorphic and hence the conclusion follows.

\bibliographystyle{alpha}
\bibliography{abdc}

%\backmatter

\end{document}